\documentclass[11pt,reqno]{amsart}

\usepackage{amsmath, amsthm, amssymb, esint}
\usepackage{amsfonts}
\usepackage{bbm}
\usepackage[ansinew]{inputenc}
\usepackage[dvips]{epsfig}
\usepackage{graphicx}
\usepackage[english]{babel}
\usepackage{enumerate}
\usepackage{enumitem}
\usepackage{hyperref}
\usepackage{fbox}
\usepackage{mathrsfs}
\usepackage{color}
 \usepackage{accents}

\theoremstyle{plain}
\newtheorem{thm}{Theorem}[section]
\newtheorem{cor}[thm]{Corollary}
\newtheorem{lem}[thm]{Lemma}
\newtheorem{prop}[thm]{Proposition}

\newtheorem{pb}[thm]{Problem}

\theoremstyle{definition}
\newtheorem{defi}[thm]{Definition}

\theoremstyle{remark}
\newtheorem{rem}[thm]{Remark}

\numberwithin{equation}{section}

\newcommand{\de}{\partial}

\newcommand{\R}{\mathbb{R}}

\newcommand{\N}{\mathbb{N}}

\newcommand{\eps}{\varepsilon}

\newcommand{\be}{{\boldsymbol{e}}}

\newcommand{\PosS}{\{u>0\}}
\def\vem{\vspace{1em}}
\def\CSetC{\Lambda(u)}
\def\CSetT{\Lambda(u)}

\newcommand{\average}{{\mathchoice {\kern1ex\vcenter{\hrule height.4pt
width 6pt depth0pt} \kern-9.7pt} {\kern1ex\vcenter{\hrule
height.4pt width 4.3pt depth0pt} \kern-7pt} {} {} }}
\newcommand{\ave}{\average\int}
\def\R{\mathbb{R}}

\DeclareMathOperator*{\argmin}{arg\,min}

\title[Graphical Solutions for One-Phase Free Boundaries]{Graphical Solutions to  One-Phase \\Free Boundary Problems}

\author{Max Engelstein}
\address{127 Vincent Hall 206 Church St. SE Minneapolis, MN, USA 55455}
\email{mengelst@umn.edu}

\author{Xavier Fern\'andez-Real}
\address{EPFL SB, Station 8, CH-1015 Lausanne, Switzerland}
\email{xavier.fernandez-real@epfl.ch}

\author{Hui Yu}
\address{Department of Mathematics, National University of Singapore, Singapore 119076}
\email{huiyu@nus.edu.sg}



\thanks{M.E. was partially supported by NSF DMS 2000288 and NSF CAREER 2143719. X.F. was supported by the Swiss National Science Foundation (SNF grants 200021\_182565 and PZ00P2\_208930),   by the Swiss State Secretariat for Education, Research and lnnovation (SERI) under contract number MB22.00034, and by the AEI project PID2021-125021NAI00 (Spain). H.Y. was supported by the Presidential Young Professor Fund (National University of Singapore). All three authors would like to thank the anonymous referee for their careful reading and many comments which improved the manuscript.}

\begin{document}

\begin{abstract}
We study viscosity solutions to the classical one-phase problem and its thin counterpart. In low dimensions, we show that when the free boundary is the graph of a continuous function, the solution is the half-plane solution. This answers, in the salient dimensions, a one-phase free boundary analogue of Bernstein's problem for minimal surfaces. 

As an application, we also classify monotone solutions of semilinear equations with a  bump-type nonlinearity.  
\end{abstract}

\maketitle 

\tableofcontents

\section{Introduction}
In this work, we deal with the Bernoulli free boundary problem in both the classical formulation, also known as the \textit{classical one-phase problem},
\begin{equation}
\label{eq:viscosity_sol}
\left\{
\begin{array}{rcll}
u   &\ge& 0 &\quad\text{in}\quad \Omega\subset\R^n,\\
\Delta u &=& 0&\quad\text{in}\quad \{u > 0\}\cap\Omega,\\
|\nabla u| &=& 1&\quad\text{on}\quad \partial\{u > 0\}\cap\Omega,
\end{array}
\right.
\end{equation}
and the thin formulation, also known as the \textit{thin one-phase problem},
\begin{equation}
\label{eq:visc_sol_s}
\left\{
\begin{array}{rcll}
u &\ge& 0 &\quad\text{in}\quad \Omega\subset\R^{n+1},\\
\Delta u &=& 0&\quad\text{in}\quad \{u > 0\}\cap\Omega,\\
\partial^{1/2}_\nu u &=& 1 &\quad\text{on}\quad \de\{u>0\}\cap\{x_{n+1}=0\}\cap\Omega.
\end{array}
\right.
\end{equation}
Here the `half-normal derivative' $\partial^{1/2}_\nu u$ is defined as
\[
\partial^{1/2}_\nu u(z) := \lim_{t\downarrow 0}\, t^{-1/2} u(z+t\nu(z)),
\]
where $\nu\in \mathbb{S}^{n}\cap \{x_{n+1} = 0\}$ is the inner normal vector along the free boundary, 
\[
\partial_{\R^n}\left( \{u>0\}\cap\{x_{n+1}=0\}\right).
\]
 In each case, the solution $u$ is a continuous function satisfying the equations in the \textit{viscosity sense}. For the precise definitions of viscosity solutions, see Definitions \ref{DefViscositySolutionsClassical} and \ref{DefViscositySolutionThin}.

For the classical one-phase problem, the zero level set of the solution is sometimes referred to as the \textit{contact set}, namely, 
\begin{equation}
\label{ContactSetClassical}
\Lambda(u):=\{u=0\}.
\end{equation}
For the thin version, the contact set is contained inside a lower-dimensional subspace,
\begin{equation}
\label{ContactSetThin}
\Lambda(u):=\{u=0\}\cap\{x_{n+1}=0\}.
\end{equation}
Outside the contact sets, the solutions are harmonic. Along the boundary of the contact sets, the so-called \textit{free boundaries}, both the value and the rate of change of the solutions are prescribed, leading to an overdetermined problem. As such, not every set can be the free boundary of a solution, and to understand a solution it (essentially) suffices to understand the free boundary. 

\vem 

There has been a lot of research devoted to understanding the free boundary of both \ref{eq:viscosity_sol} and \ref{eq:visc_sol_s} (see below for more details), of which one important aspect is the classification of solutions in the entire space. 

Such a classification has recently been completed for the obstacle problem (another free boundary problem) by Eberle--Figalli--Weiss \cite{EFW22} (see also \cite{ESW20, EY23b}), concluding a program that lasted for more than 90 years (this classification also has implications for the fine properties of free boundaries, c.f.~\cite{ESWPreprint}).  For the thin obstacle problem, a partial classification has been achieved in \cite{ERW21, EY23}. In both cases, the results state that, under some restrictions,  the space of entire solutions is finite-dimensional.

The obstacle problem and its thin counterpart  arise as Euler--Lagrange equations of convex energy functionals. The convexity of the functionals implies that  viscosity solutions are minimizers of the energy, allowing the usage of both variational and nonvariational techniques. 

For our problems \eqref{eq:viscosity_sol} and \eqref{eq:visc_sol_s}, however, the underlying functionals are not convex, and the spaces of viscosity solutions are much wider than minimizers of the functionals (see Definition \ref{DefMinimizersClassical}  and Definition \ref{DefMinimizersthin} for the definitions of minimizers). Indeed, the original motivation for the viscosity framework is to construct non-minimizing solutions \cite{Caf88}, which show up naturally in domain variation problems \cite{HHP11} and fluid mechanics \cite{BSS76, CG11}. 

This flexibility of the viscosity framework allows a wide-range of behaviors, and some important energy-based tools are no longer available (for instance, the nondegeneracy property may not hold for general viscosity solutions, \cite{KW22}).  As a result, even in two dimensions, the best classification result for smooth solutions to the classical one-phase problem requires topological restrictions \cite{T14, JK16}. For the classical one-phase problem in higher dimensions, or for the thin one-phase problem, a full classification of entire solutions seems out of reach. This can be thought of in analogy with globally defined minimal surfaces, for which a plethora of examples exist in $\mathbb R^3$, but there is no complete list (see, e.g. \cite{CM11}). 

\vem 

As a starting point for this classification, we propose to study solutions to \eqref{eq:viscosity_sol} and \eqref{eq:visc_sol_s} with graphical free boundaries. To be precise, we study viscosity solutions  whose contact sets (see \eqref{ContactSetClassical} and \eqref{ContactSetThin}) are subgraphs of continuous functions. 
Under this topological assumption, we show that viscosity solutions are minimizers for the underlying energy functionals (a result which may be of independent interest). This allows us to classify, in low dimensions, the space of entire viscosity solutions with graphical free boundaries. 

Our approach is inspired by the Bernstein conjecture for minimal surfaces, which states that the only graphical minimal surface is the hyperplane \cite{Ber15}. It was shown that $n$-dimensional minimal graphs in  $\mathbb R^{n+1}$ must be hyperplanes for $n \leq 7$ (see \cite{Fle62,DeG65, Alm66, Sim68}); while in higher dimensions, it is false by an example given in \cite{BDG69}. Similarly, we do not expect our results to hold in higher dimensions (large enough to allow for singular minimizers), though no analogue to the construction in \cite{BDG69} has been found for \eqref{eq:viscosity_sol} or \eqref{eq:visc_sol_s}.

In the following, we describe our results in the classical regime \eqref{eq:viscosity_sol} in subsection~\ref{sec:classical}, and in the thin regime \eqref{eq:visc_sol_s} in subsection~\ref{sec:thinreg}.

\subsection{The classical regime}
\label{sec:classical}
The classical one-phase problem \eqref{eq:viscosity_sol} arises as the Euler-Lagrange equation to the \textit{Alt--Caffarelli functional}
\begin{equation}
\label{eq:JOm}
\mathcal{J}_\Omega(v) = \int_\Omega |\nabla v|^2 + |\{v > 0\}|,\qquad\text{for}\quad v \ge 0,\quad v\in H^1(\Omega),
\end{equation}
where $\Omega$ is a domain in $\R^n$.

Motivated by models in flame propagation and jet flows \cite{BL82, ACF82, ACF82b, ACF83,CV95}, this energy was originally studied from a mathematical point of view by Alt and Caffarelli in \cite{AC81}. Since then, regularity  of the minimizer and its free boundary has been extensively studied, see, for instance,  \cite{AC81, Caf87, D11, ESV20, FY23}. We refer to \cite{CS05} for a thorough introduction to the classical theory, and refer to \cite{Vel19} for a modern treatment of the one-phase problem and related topics. 

\vem

Even homogeneous minimizers of \eqref{eq:JOm} (also known as \textit{minimizing cones}) have not been fully classified.  By the works of Caffarelli--Jerison--Kenig \cite{CJK04} and Jerison--Savin \cite{JS15}, it is known that for $n\le 4$, the only homogeneous minimizer\footnote{The result applies to a larger class called \emph{stable solutions}. They are critical points of the functional \eqref{eq:JOm} with nonnegative second variations.} is, up to a rotation, the \textit{half-plane solution}
\begin{equation}\label{HalfSpaceSolution}
u(x)=x_n^+.
\end{equation}
While in dimension $7$, De Silva-Jerison \cite{DJ09} provides a nonflat minimizing cone.

The largest dimension in which homogeneous minimizers must be flat is currently unknown. In this work, we denote the largest such dimension by $n_{\rm local}^*$, that is,
\begin{equation}
\label{LargestDimensionClassical}
n_{\rm local}^*:=\max\{n: \text{minimizing cones of }\eqref{eq:JOm} \text{ in $\R^n$ are rotations of }\eqref{HalfSpaceSolution}\}.
\end{equation} 
With the aforementioned works, we have
$$
4\le n_{\rm local}^*\le6.
$$

Without assuming homogeneity, minimizers exhibit even richer behavior. For instance, associated with each nonflat minimizing cone, there is a family of minimizers whose free boundaries foliate the entire space $\R^n$ (see \cite{DJS22, ESV22}).

\vem

Solutions to the one-phase problem \eqref{eq:viscosity_sol} that are not minimizers of the Alt--Caffarelli energy functional \eqref{eq:JOm} arise naturally in problems involving domain variations \cite{HHP11} and fluid mechanics \cite{BSS76,CG11}. In these contexts, the positive set, $\PosS$, of a solution $u$ in the entire space $\R^n$  is sometimes referred to as an \textit{exceptional domain}. The classification of exceptional domains is an important topic that so far has been successful only for  special classes of domains.

With the half-plane solution from \eqref{HalfSpaceSolution}, we see that the half-plane $\{x_n>0\}$ is an exceptional domain. The union of two half-planes, $\{x_n>0\}\cup\{x_n<-a\}$ with $a\ge0$, is also an exceptional domain corresponding to the solution $u=x_n^++(x_n+a)^-$. By taking a truncation of the fundamental solution, we see that the exterior of the ball $\R^n\setminus B_R$ is an exceptional domain if $R>0$ is chosen properly. Apart from these classical examples, a family of catenoid-like domains were discovered by Hauswirth--H\'elein--Pacard \cite{HHP11} in the plane, and by Liu--Wang--Wei \cite{LWW21} in general dimensions. A family of periodic exceptional domains appeared in \cite{BSS76}.

By the work of Traizet \cite{T14}, we know that in the plane, these are all the exceptional domains whose boundaries are smooth and have finitely many components.  A similar result was obtained by Khavinson--Lundberg--Teodorescu \cite{KLT13}, who also showed that in general dimensions, the exterior of a ball is the only smooth exceptional domain with bounded complement. 

\vem


In the first part of this work, we deal with viscosity solutions to \eqref{eq:viscosity_sol} with graphical free boundaries. Concerning these solutions, our first main result states:
\begin{thm}
\label{thm:main0}
Let  $u$ be a viscosity solution to the classical one-phase problem \eqref{eq:viscosity_sol} in $\R^n$ for
$$n\le n_{\rm local}^*+1.$$ 

If its contact set $\Lambda(u)$ is the subgraph of a continuous function,
then we have
$$
u = x_n^+
$$
up to a rotation and a translation.
\end{thm}
Recall the critical dimension $n_{\rm local}^*$ and the contact set $\Lambda(u)$ defined in \eqref{LargestDimensionClassical} and \eqref{ContactSetClassical} respectively. 

\begin{rem}
For smooth $\partial\{u>0\}$ in $\R^2$, Hauswirth--H\'elein--Pacard showed a similar result  (substituting the assumption on the contact set with the related assumption of monotonicity in a direction) in \cite{HHP11} with complex variable techniques. 
\end{rem}

\begin{rem}\label{rem:assumptions}
While we do not claim the condition requiring the graph to be continuous is sharp, some regularity assumption is necessary on the graphical free boundary. 

Indeed, taking $u(x_1,x_2)$ to be any solution in $\R^2$ (for instance, the catenoid-type solution in \cite{HHP11}), we can extend it  to $\R^3$ trivially as $\overline{u}(x_1,x_2,x_3):=u(x_1,x_2)$. For such a function, its contact set is the subgraph of a (generalized) function of the form $x_3=\varphi(x_1,x_2)$ with $\varphi=-\infty$ in $\PosS$ and $\varphi=+\infty$ in $\{u=0\}$.

\end{rem}

\vem 
To prove Theorem \ref{thm:main0}, the natural idea is to  reduce the problem to the study of homogeneous solutions by a blow-down procedure. Unfortunately, due to the lack of variational tools 
(monotonicity formula, nondegeneracy property, etc.), 
a blow-down analysis for general viscosity solutions seems difficult.

For the class of solutions we are considering, however, we can show they are actually minimizers of the Alt--Caffarelli energy \eqref{eq:JOm}. This is one of our main technical contributions to the classical one-phase problem and should be of independent interest (see, e.g. the discussion in the introduction of \cite{DJ11}):

\begin{thm}
\label{MinimizingClassicalIntro}
Suppose that $u$  is a viscosity solution to the classical one-phase problem \eqref{eq:viscosity_sol} in $\R^n$, and that its contact set $\CSetC$ is the subgraph of a continuous function. 

Then $u$ is a global minimizer of the Alt--Caffarelli energy \eqref{eq:JOm}. 
\end{thm} 

For the definition of a global minimizer, see Definition \ref{DefMinimizersClassical}.

\begin{rem}See Proposition \ref{prop:cont_graph}
for a localized version of this theorem. \end{rem}

While this theorem is inspired by a similar result for graphical minimal surfaces (or for strictly monotone solutions to semilinear equations), in our case the proof is more delicate. 

Indeed, for a graphical minimal surface, its minimizing property can be established by a standard sliding argument. To be precise, for a function $\varphi$ satisfying the minimal surface equation in $\R^n$, we need to show that its graph, to be denoted by $\Gamma_\varphi$, minimizes the area over surfaces with the same boundary data. Suppose not: we find $B_R\subset\R^n$ and a surface $M$ which matches $\Gamma_\varphi$ along $\partial B_R\times\R$ and has strictly less area.  Without loss of generality, we may assume $M$ is a minimizer of the area with given boundary data.

Now we translate $\Gamma_\varphi$ vertically. With $\Gamma_\varphi\neq M$ in $B_R\times\R$, 
there is a critical instant when $\Gamma_\varphi$ lies on one side of $M$ but $\Gamma_\varphi\cap M$ is nonempty. Since the two surfaces are translations of one another along $\partial B_R \times \mathbb R$, the point of intersection can be found in the interior of the domain. This contradicts the strict maximum principle between minimal surfaces. 

To implement a similar strategy in our context, there are several challenges. 

Firstly, our problem involves not only the free boundary but also the solution. To perform the sliding argument, we need to translate a comparison between the free boundaries into a comparison between the associated solutions. This is achieved by showing that the graphicality assumption implies the monotonicity of the solution (See Proposition \ref{prop:cont_graph}). 

Secondly, while graphical minimal surfaces instantly regularize  in the interior of the domain, see \cite{BG72}, a similar property for graphical free boundaries (in fact, for monotone solutions) holds when assuming the minimizing property (in fact under the weaker assumption that the positivity set has some quantitative topology) \cite{DJ11}, which is what we need to prove. This lack of regularity for free boundaries also means the comparison principle is much weaker. Even among minimizers of the Alt--Caffarelli functional, a strict maximum principle has only recently been established in \cite{ESV22}. For viscosity solutions, such a result is not known. We overcome this difficulty by working with sup/inf-convolutions instead of the original solution. 

The last challenge we need to overcome is the `boundary stickiness' pheno\-menon, that is, a large portion of the positive set $\PosS$ of a minimizer `invades' the zero region on the fixed boundary. For instance, suppose that $u$ is a minimizer in the two-dimensional domain $\{(x_1,x_2):x_1\in(-1,1),x_2\in(0,\delta)\}$ with $u=1$ on $\{x_2=\delta\}$, and $u=0$ on the remaining parts of the boundary. By choosing $\delta$ small, it can be shown that $u$ will be positive in the entire domain. When this happens, the free boundary is `stuck' to the fixed boundary in some sense, and the sliding argument described above could fail due to contact points along the fixed boundary. To rule out this possibility, we need precise information about the separation of the free boundary from the fixed boundary. Fortunately for us, this result has recently been obtained by Chang-Lara and Savin \cite{CS19}, allowing us to complete the proof of Theorem \ref{MinimizingClassicalIntro}. 

\vem

With Theorem \ref{MinimizingClassicalIntro} in hand, we can perform a blow-down analysis of the solution $u$ to obtain a minimizing cone $u_\infty$  in $\R^n$. With $n\le n_{\rm local}^*+1$, its free boundary has smooth trace on the sphere $\mathbb{S}^{n-1}$ (here is where we use the restriction on the dimension). Being the limit of graphical solutions, this cone $u_\infty$ is also graphical. A maximum principle type argument, applied to the directional derivatives of $u_\infty$, implies that $u_\infty$ is a half-plane solution. 

This means that our original solution $u$ is `flat at large scales'. An improvement of flatness argument as in \cite{D11} gives the desired flatness of $u$. 
 
\subsection{Application to semilinear equations} 
%
%
%
%
%
De Giorgi conjectured in 1978, \cite{DeG78}, that monotone solutions (critical points) $u$ of the Ginzburg--Landau energy (alternatively,   solutions to the Allen--Cahn equation) 
\[
 \Delta u = -u(1-u^2)\quad\text{in}\quad \R^n,
\]
with $\|u\|_{L^\infty(\R^n)}\le 1$, must have one-dimensional symmetry (alternatively, all level sets must be hyperplanes) in $\R^n$ with $n \le 8$. This is currently known as \emph{De Giorgi's conjecture}. It was proven to hold in a series of papers in dimensions 2 and 3, \cite{GG98, AC00}, that culminated with the remarkable work by Savin \cite{Sav09} for $4\le n\le 8$, where it was shown under the additional assumption
\begin{equation}
\label{eq:additional}
\lim_{x_n\to \pm \infty} u(x', x_n) = \pm 1,
\end{equation}
which ensures that solutions are  minimizers to the corresponding energy.   A counter-example when $n\geq 9$ was constructed in \cite{DKW11}.  The paper \cite{Sav09} also applies to general solutions to semilinear equations $ \Delta u = f(u)$ in $ \R^n$, provided that  $f$ is the derivative of a ``double-well potential" (with wells of the ``same depth''). This established a relation between the study of minimal surfaces (and in particular, entire minimal graphs) and solutions to semilinear equations arising from local minimizers of an energy (for $f$ coming from double-well potentials; in particular, $f$ with zero integral in the range of $u$). 

For other types of semilinear equations (namely, those where $f$ is \emph{similar} to a bump function or a Dirac delta; alternatively, when $f$ has nonzero and finite integral in the range of $u$) the corresponding analogy is not with minimal surfaces, but instead, with the one-phase problem (see \cite{CS05, FR19, AS22}). In particular, under the appropriate scaling of \emph{non-double-well} potential functionals, the corresponding limits are solutions to the one-phase problem, and hence the corresponding zero-level set converges to the free boundary of a one-phase problem. This relation was already observed in \cite{CS05}, and then studied in \cite{FR19} to classify global solutions, and more recently in \cite{AS22} to obtain a classification of global minimizers to semilinear equations with $f$ of \emph{bump type}. 

In analogy with De Giorgi's conjecture, we have 
\begin{pb}
\label{pb:B}
Let $u$ satisfy $ \Delta u = f(u)$ in $\R^n$ for some $f$ of \emph{bump type} and $\partial_{x_n} u > 0$ in $\R^n$. 

If $n \le n_{\rm local}^* +1$, then $u$ is a one-dimensional solution. 
\end{pb}

Here, we say that $f$ is of \emph{bump type} if $f\ge 0$, $f(0) = 0$, $f'(0) > 0$ and $\int_0^\infty f = 1$; these are the types of semilinear equations studied in \cite{FR19, AS22}.

As a consequence of our previous result, and thanks to \cite{AS22}, we prove that Problem~\ref{pb:B} is true under the following additional growth assumption (in analogy with \eqref{eq:additional}):
\begin{equation}
\label{eq:cond_min}
\lim_{x_n\to -\infty} u(x', x_n) = 0\quad\text{and}\quad \lim_{x_n\to +\infty} u(x', x_n) = \infty. 
\end{equation}

  Thus, we have:

\begin{cor}\label{cor:semilinear}
Problem~\ref{pb:B} holds with the additional assumption \eqref{eq:cond_min}. 
\end{cor}

\subsection{The thin regime}
\label{sec:thinreg}
The thin one-phase problem \eqref{eq:visc_sol_s} corresponds to the Euler--Lagrange equation of the  \textit{thin one-phase energy functional}. 
Given a domain $\Omega\subset \R^{n+1}$ that is even in the last variable\footnote{The evenness of the domain, the function and/or the boundary conditions is a natural assumption for this problem which we will make throughout and is shared by most of the literature. We mention here only that it comes out of a connection to a nonlocal free boundary problem in the thin-space $\{x_{n+1} =0\}$ and encourage the reader to look into the introductions of \cite{DR12, CRS10, EKPSS20} for more background and information}, and denoting 
\begin{equation}
\label{DecompOfVariable}
x = (x', y)\in \R^n\times \R,
\end{equation} 
we define:
\begin{equation}
\label{eq:JOm_frac}
\mathcal{J}^{0}_\Omega(v) = \int_\Omega |\nabla v|^2 \, dx + \lambda\mathcal{H}^n\left(\{v > 0\}\cap \{y = 0\}\cap \Omega\right),\quad\text{for}~~v \ge 0,~~v\in H^1(\Omega),
\end{equation}
where $\mathcal{H}^n$ denotes the $n$-dimensional Hausdorff measure, and $\lambda>0$ is a universal constant\footnote{This constant is chosen so that the free boundary condition in \eqref{eq:visc_sol_s} has value $1$ as the right-hand side.}.

This functional was introduced by Caffarelli--Roquejoffre--Sire to address certain phenomena in plasma physics and semi-conductor theory that involve long-range interactions \cite{CRS10}. Since then, the regularity of minimizers of \eqref{eq:JOm_frac} as well as viscosity solutions to \eqref{eq:visc_sol_s} has been studied extensively. See, for instance, \cite{DR12, DS12, DSS14, EKPSS20}.

Just as in the classical case, the classification of homogeneous minimizers/ minimizing cones remains an important open question for the thin one-phase problem. For this problem, the corresponding \textit{half-plane solution} is  
\begin{equation}
\label{HalfSpaceSolutionThin}
u(x) = U(x_n, y) := \frac{1}{\sqrt2}\sqrt{x_n+\sqrt{x_n^2+y^2}}.
\end{equation} 
This is shown to be the only minimizing cone in dimension $2+1$, \cite{DS15b}. If we assume axial-symmetry of the cones, nonflat minimizing cones\footnote{The result in \cite{FR20} rules out stable cones, that is, those cones with nonnegative second variation for \eqref{eq:JOm_frac}.}
can be ruled out in dimensions $n+1\le 6$, \cite{FR20}.

The half-plane solution is expected to be the only minimizing cone in low dimensions. However, it is currently unknown what  the critical dimension is.  
 In this work, we denote it by $n_{\rm thin}^*$, that is,
\begin{equation}
\label{LargestDimensionThin}
n_{\rm thin}^*:=\max\{n: \text{minimizing cones of }\eqref{eq:JOm_frac} \text{ are rotations of }\eqref{HalfSpaceSolutionThin} \text{ in }\R^{n+1}\}.
\end{equation} 

For the classification of entire viscosity solutions, even less is known. To the knowledge of the authors, the only result available is in \cite[Proposition 6.4] {DS15b}. That result states that for a homogeneous viscosity solution $u$, if its contact set~$\Lambda(u)$ (see~\eqref{ContactSetThin}) is the subgraph of a Lipschitz function, then $u$ must be a half-plane solution.  

\vem

In dimensions lower than $n^*_{\rm local}+1$, our main result in the thin case removes the assumption on homogeneity and Lipschitz regularity of the free boundary:
\begin{thm}
\label{thm:mainThin}
Let  $u$ be a viscosity solution to the thin one-phase problem \eqref{eq:visc_sol_s} in $\R^{n+1}$ with 
$$
n\le n^*_{\rm thin}+1.
$$
If its contact set $\CSetT$ is the subgraph of a continuous function on $\{x_{n+1} = 0\}$, then
$$
u=\frac{1}{\sqrt2}\sqrt{x_n+\sqrt{x_n^2+y^2}}
$$
up to a rotation and a translation.
\end{thm} 

Similar to the classical case, it remains to be seen what the sharp assumption on the regularity of the free boundary is, see Remark \ref{rem:assumptions}.

The key ingredient in  the proof of Theorem \ref{thm:mainThin} is again the variational structure provided by the graphicality assumption, namely, 
\begin{thm}
\label{MinimizingThinIntro}
Let  $u$ be a viscosity solution to the thin one-phase problem \eqref{eq:visc_sol_s} in $\R^{n+1}$ whose contact set $\CSetT$ is the subgraph of a continuous function on $\{x_{n+1} = 0\}$.

Then $u$ is a global  minimizer of  the thin one-phase energy \eqref{eq:JOm_frac}.
\end{thm} 
See Definition \ref{DefMinimizersthin} for the definition of  a global minimizer.
\begin{rem}
See Proposition \ref{prop:cont_graph_s} for a localized version of this result. 
\end{rem} 
\begin{rem}
See also \cite{CEF22}, where the authors prove, by constructing a new nonlocal calibration functional, that strictly monotone (bounded) solutions to semilinear nonlocal equations are minimizers of the corresponding functional.
 \end{rem}
The challenges we described after Theorem \ref{MinimizingClassicalIntro} are still present for the thin case, and most can be overcome with similar strategies. The issue of `boundary stickiness', however, requires new ideas, as the boundary behavior of minimizers, in the sense of Chang-Lara and Savin \cite{CS19}, has not been studied in the thin case.  We address this in the following theorem, which may be of independent interest:

\begin{thm}
\label{BoundaryRegularityIntro}
Let $w$ be a minimizer of the thin one-phase energy \eqref{eq:JOm_frac} in 
$$
\Omega=B_1\cap\{x_1\ge0\}\subset\R^{n+1}
$$
with 
$$
w=\psi \text{ on $B_1\cap\{x_1=0\}$}.
$$

If we assume that 
$$
\psi\in C^{1/2}(\{x_1=0\}) \text{ and }\psi=0 \text{ on $\{y=0\}$,}
$$
then we have 
$w\in C^{1/2}(B_{1/2}\cap\{x_1\ge0\})$ 
with
\[
\|w\|_{C^{1/2}(B_{1/2}\cap \{x_1\ge 0\})}\le C\left(\|\psi\|_{C^{1/2}(B_1\cap \{x_1 = 0\})}+\|w\|_{L^\infty(B_1\cap \{x_1\ge0\})}+1\right)
\]
for some $C$ depending only on $n$. 

If we further assume that 
$$
\psi\le\omega(|y|)|y|^{1/2} \text{ on $\{x_1=0\}$}
$$
for some modulus of continuity $\omega$, 
then for each $x\in B_{1/2}\cap\{y=0\}\cap\overline{\{w>0\}}$ and $r\in(0,1/2)$, 
we have
$$
\sup_{B_r(x)\cap\{x_1\ge0\}}w\ge cr^{1/2}
$$
for some $c$ depending only on $n$ and $\omega$.
\end{thm} 
Recall our convention for the coordinate system in $\R^{n+1}$ from \eqref{DecompOfVariable}.

With Theorem \ref{BoundaryRegularityIntro}, we establish Theorem \ref{MinimizingThinIntro}, which allows us to use tools based on the variational structure of the problem (monotonicity formula and nondegeneracy, etc). This reduces the problem to the study of homogeneous minimizers, and finally gives our classification of graphical viscosity solutions in low dimensions as in Theorem \ref{thm:mainThin}. 
\subsection{Structure of the paper}
In Sections \ref{sec:2} to \ref{sec:4}, we study the classical one-phase problem \eqref{eq:viscosity_sol}. 
In Section \ref{sec:2}, we recall some preliminary results and introduce some notations. In Section \ref{sec:3}, we show that graphical solutions are minimizers as stated in Theorem \ref{MinimizingClassicalIntro}. In Section \ref{sec:4}, we complete the classification of graphical solutions in low dimensions and prove Theorem \ref{thm:main0}.

We deal with the thin one-phase problem \eqref{eq:visc_sol_s} in Sections \ref{sec:5} to \ref{sec:8}. Our structure parallels the classical treatment. Section~\ref{sec:5} is devoted to some preliminaries and notations. In Section~\ref{sec:6}, we show that monotone solutions are minimizers, as stated in Theorem \ref{MinimizingThinIntro}, assuming Theorem~\ref{BoundaryRegularityIntro}.  Section~\ref{sec:7} is devoted to the blow-down analysis and the classification of graphical minimizing solutions in low dimensions. Finally, in Section~\ref{sec:8}, we prove Theorem~\ref{BoundaryRegularityIntro}.
\addtocontents{toc}{\protect\setcounter{tocdepth}{0}}
\section*{Acknowledgements}
This paper was finished while the first and third authors were in residence at Institut Mittag-Leffler for the program on ``Geometric Aspects of Nonlinear Partial Differential Equations". They thank the institute for its hospitality. 

The authors would also like to thank Yash Jhaveri for fruitful discussions on the topics of this paper.  
\addtocontents{toc}{\protect\setcounter{tocdepth}{1}}

\section{Preliminaries and notations: the classical regime}
\label{sec:2}

In this section, we collect some preliminary facts about solutions to the classical one-phase problem \eqref{eq:viscosity_sol}. 

 We begin with the   definition of  viscosity solutions to \eqref{eq:viscosity_sol} as in Caffarelli--Salsa \cite{CS05} (cf. also with \cite{Caf89}). To do that, we first introduce \emph{comparison solutions}, that will work as test functions:
\begin{defi}
\label{defi:comp_sol}
Let $u\in C(\Omega)$ for some domain $\Omega\subset \R^n$, $u \ge 0$ in $\Omega$. 
\begin{enumerate}[leftmargin=*, label=(\roman*)]
\item We say that $u$ is a (strict) \emph{comparison subsolution} to the classical one-phase problem \eqref{eq:viscosity_sol} if \[
u\in C^2(\PosS),\qquad \Delta u \ge 0\quad\text{ in }\quad \PosS,\]
  the free boundary $\partial\PosS$ is a $C^2$ manifold, and for any $x_\circ \in \partial\PosS$ we have 
  \[
  u_\nu(x_\circ) := \nu \cdot \nabla u(x_\circ)> 1,
  \] where $\nu\in \mathbb{S}^{n-1}$ is the inward normal to $\partial\PosS$ at $x_\circ$, $\nu = \frac{\nabla u}{|\nabla u|}(x_\circ)$. 
  
  \item We say that $u$ is a (strict) \emph{comparison supersolution} to the classical one-phase problem \eqref{eq:viscosity_sol} if \[
u\in C^2(\PosS),\qquad \Delta u \le 0\quad\text{ in }\quad \PosS,\]
  the free boundary $\partial\PosS$ is a $C^2$ manifold. and  for any $x_\circ \in \partial\PosS$  we have 
  \[
  u_\nu(x_\circ) < 1,
  \] where $\nu\in \mathbb{S}^{n-1}$ is the inward normal to $\partial\PosS$ at $x_\circ$, $\nu = \frac{\nabla u}{|\nabla u|}(x_\circ)$. 
\end{enumerate}
\end{defi}

By means of the previous definition, we can introduce the notion of a \emph{viscosity solution}: 
\begin{defi}
\label{DefViscositySolutionsClassical}
Let $u\in C(\Omega)$ for some domain $\Omega\subset \R^n$, $u \ge 0$ in $\Omega$. We say that $u$ is a \emph{viscosity solution} to the classical one-phase problem \eqref{eq:viscosity_sol} if  
$$
\Delta u = 0\quad \text{ in }\quad \{u>0\}\cap\Omega,
$$ 
and any strict comparison subsolution (resp. supersolution) cannot touch $u$ from below (resp. from above) at a free boundary point $x_\circ\in \partial\PosS$. 
\end{defi} 

In the previous definition, we say that a strict comparison subsolution $v$ touches from below $u$ at a free boundary point $x_\circ\in \partial\PosS$ if $x_\circ\in \partial\{v > 0\}$ and $v \le u$ in a neighborhood of $x_\circ$.

Unless otherwise specified, solutions should always be understood in the viscosity sense in the remaining part of the paper.  In general, singularities are inevitable on the free boundary of a viscosity solution. To use various comparison principles, it is often necessary to regularize the free boundary first. To this end, sup/inf-convolutions are powerful technical tools.

\begin{defi}
\label{DefConvolutionsClassical}
For a domain $\Omega\subset\R^n$ and $t>0$, define
$$
\Omega_t:=\{x\in\Omega:\mathrm{dist}(x,\Omega^c)>t\}.
$$
For $u\in C(\Omega),$ its \textit{$t$-sup-convolution} is defined as
$$
\overline{u}_t(x):=\sup_{B_t(x)}u \text{ for $x\in\Omega_t$.}
$$
Its \textit{$t$-inf-convolution} is defined as
$$
\underline{u}_t(x):=\inf_{B_t(x)} u \text{ for $x\in\Omega_t$.}
$$
\end{defi} 

The following lemma motivates the use of sup/inf-convolutions. We refer to Section 2.3 of \cite{CS05}.

\begin{lem}
\label{PropertyOfConvolutionsClassical}
Let $u\in C(\Omega)$ be a viscosity solution to the classical one-phase problem \eqref{eq:viscosity_sol} in $\Omega$. For $t>0$, let $\overline{u}_t$ and $\underline{u}_t$ denote its sup-convolution and inf-convolution as in Definition \ref{DefConvolutionsClassical}. Then:

\begin{itemize}
\item $\overline{u}_t$ satisfies $\Delta\overline{u}_t \ge 0$ in $\{\overline{u}_t > 0\}\cap \Omega_t$ and, for each $x_\circ\in \partial \{\overline{u}_t > 0\}\cap \Omega_t$, there is a point $p$ such that 
\[
B_t(p) \subset \{\overline{u}_t > 0\} \quad\text{and}\quad x_\circ \in \partial B_t(p),
\]
and 
\[
\overline{u}_t(x) \ge \langle x-x_\circ, \nu\rangle_+ + o(|x-x_\circ|)
\]
near $x_\circ$, where $\nu :=\frac{1}{t}(p-x_\circ)$.
\item $\underline{u}_t$ satisfies $\Delta\underline{u}_t \le 0$ in $\{\underline{u}_t > 0\}\cap \Omega_t$ and, for each $x_\circ\in \partial \{\underline{u}_t > 0\}\cap \Omega_t$, there is a point $p$ such that 
\[
B_t(p) \subset \{\underline{u}_t = 0\} \quad\text{and}\quad x_\circ \in \partial B_t(p),
\]
and 
\[
\underline{u}_t(x) \le \langle x-x_\circ, \nu\rangle_+ + o(|x-x_\circ|)
\]
near $x_\circ$, where $\nu :=\frac{1}{t}(x_\circ-p)$.
\end{itemize}
\end{lem} 

We now turn to some well known results regarding the regularity of viscosity solutions. First we recall that, viscosity solutions in the entire space $\R^n$ have a dimensional gradient bound:
\begin{lem}\label{lem:visclip}
Let $u$ be a viscosity solution in $\R^n$ to the classical one-phase problem. Then, there is a dimensional constant $C$ such that 
$$
|\nabla u|\le C \text{ in $\R^n$.}
$$
\end{lem}
For a proof, see, for instance, Lemma 11.19 of \cite{CS05}.

A fundamental tool in the study of the one-phase problem is the following improvement-of-flatness lemma from \cite{D11}. We will use it at large scales to classify entire solutions in low dimensions.
\begin{lem}
\label{LemIOFClassical}
Suppose that $u$ is a solution to the classical one-phase problem \eqref{eq:viscosity_sol} in $B_1$ with $0\in\partial\CSetC$ and
$$
(x_n-\eps)_+\le u\le (x_n+\eps)_+ \text{ in $B_1$.}
$$
There are dimensional constants $\bar \eps$, $r$, and $C$ such that if $\eps<\bar \eps$, then we can find $\be\in\mathbb{S}^{n-1}$ satisfying 
$$
|\be-\be_n|\le C\eps^2,
$$
and 
$$
(x\cdot \be-\eps  r/2)_+\le u(x)\le (x\cdot \be+\eps  r/2)_+ \text{ in $B_{ r}$.}
$$
\end{lem} 




A special class of solutions to the classical one-phase problem \eqref{eq:viscosity_sol} arises as the minimizers of the Alt--Caffarelli functional \eqref{eq:JOm}.
 \begin{defi}
 \label{DefMinimizersClassical}
 For $\Omega\subset\R^n$ and $u\in H^1(\Omega)$,
we say that $u$ is a \emph{minimizer} of the Alt--Caffarelli functional \eqref{eq:JOm} in $\Omega$ if 
$u\ge0$ in $\Omega$, and 
\[
\mathcal{J}_\Omega(u) \le \mathcal{J}_\Omega(v)\qquad\text{for all}\quad v\ge0, \quad v-u\in H^1_0(\Omega). 
\]

For $u\in H^1_{\rm loc}(\R^n)$ with $u\ge 0$, we say that it is a \textit{global minimizer} in $\R^n$ if
it is a minimizer in $B_R$ for every $R>0$.
\end{defi}

Compared with viscosity solutions, minimizers are particularly nice since we can apply variational tools. This allows us to perform the following blow-down argument.
\begin{lem} 
\label{lem:blowdown}
Let $u$ be a global minimizer of the Alt--Caffarelli functional in $\mathbb R^n$.
For a sequence $r_i\uparrow \infty$, define 
$$
u_i(x) := \frac{u(r_i x)}{r_i}.
$$ 
Then, perhaps passing to a subsequence, we can find a nonzero one-homogeneous global minimizer $u_\infty$ such that  
$$
u_i \rightarrow u_\infty \text{ locally uniformly in $\R^n$}
$$ 
with 
$$
\Lambda(u_i) \rightarrow \Lambda(u_\infty)\text{ in $L^1_{\mathrm{loc}}$,} 
$$
and 
$$
\partial\Lambda(u_i) \rightarrow \partial \Lambda(u_\infty) \text{ locally in the Hausdorff distance sense.}
$$ 
\end{lem}

\begin{proof}
The convergence to a nonzero global minimizer follows from the Lipschitz and nondegeneracy estimates for minimizers in \cite{AC81} (see also \cite{DT15}; this is written explicitly in \cite[Theorem 1.3]{EE}). Using the Weiss monotonicity formula and arguing as in \cite{Wei99}, the one-homogeneity of $u_\infty$ follows as long as 
$$
\lim_{r_i\uparrow \infty} W(u,r_i) < \infty,
$$
where $W$ is the Weiss energy functional. 

Towards this end, we note 
$$
W(u, R) \leq \frac{1}{R^n} \int_{B_R} |\nabla u|^2 + \chi_{\{u > 0\}}\, \le C,
$$ 
for a dimensional constant $C$,
where we used the universal gradient bound from Lemma \ref{lem:visclip}.
\end{proof}

Homogeneous minimizers have smooth free boundaries on the sphere in low dimensions. Recall the critical dimension $n_{\rm local}^*$ defined in \eqref{LargestDimensionClassical}.
\begin{lem}
\label{TraceSmoothClassical}
Let $u$ be a homogeneous minimizer in $\R^n$ with 
$$
n\le n_{\rm local}^*+1.
$$
Then $\partial\CSetC\cap\mathbb{S}^{n-1}$ is smooth.
\end{lem} 

\begin{rem}
This is the only place where we require the restriction of dimensions. 
\end{rem} 

\begin{proof}
Suppose not; then we find a singularity on $\partial\CSetC\cap\mathbb{S}^{n-1}$, say, at point $e_1$. 

Then we perform a blow-up analysis as in \cite{Wei99} and end up with a minimizer $v$, which is independent of the variable $x_1$ and has a line of singularities on the free boundary. 

By restricting $v$ to the variables $(x_2,x_3,\dots, x_n)$, we get a homogeneous minimizer in $\R^{n-1}$ with a singularity at $0$. This contradicts the definition of $n_{\rm local}^*$ as in \eqref{LargestDimensionClassical}. 
\end{proof} 

As mentioned in the introduction, one important tool we use to address the `boundary stickiness' phenomenon is the following theorem on  boundary regularity of minimizers as in \cite{CS19}. See Figure~\ref{fig:0} for a graphical representation of this setting.
\begin{thm}
\label{thm:CS}
Let $\Omega\subset \R^n$ be a domain with $C^2$ boundary, and let $Z \subset \partial\Omega$ be open with respect to the topology of $\partial \Omega$. Let $u:\overline{\Omega}\to [0,\infty)$ be a minimizer of the Alt--Caffarelli functional in $\Omega$ such that 
$$
u = 0 \text{ on $Z$}.
$$
Then $u$ solves (in the viscosity sense), 
\[
\left\{
\begin{array}{rcll}
\Delta u & = & 0& \quad\text{in}\quad \Omega^+:=\PosS\cap\Omega,\\
|\nabla u|& \ge & 1 &\quad\text{on}\quad \partial\Omega^+\cap Z,\\
|\nabla u| & = & 1&\quad\text{on}\quad\partial \Omega^+\cap \Omega. 
\end{array}
\right.
\]
Furthermore, $\partial \Omega^+$ is $C^1$ in a neighborhood of every $x_\circ \in \partial\Omega^+\cap Z$.
\end{thm}

\begin{figure}[h]
\centering
\includegraphics[scale = 1]{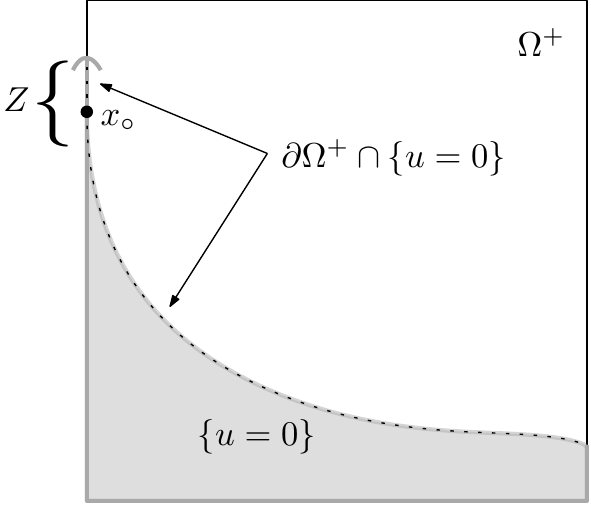}
\caption{Theorem~\ref{thm:CS} says that this is the only way in which the free boundary can detach from the fixed boundary (from the interior of the zero level set on the fixed boundary). In particular, there is always a well-defined normal at $x_\circ$.}.
\label{fig:0}
\end{figure}

\section{Graphical solutions are minimizers: the classical regime} 
\label{sec:3}

In this section, we introduce the class of solutions we are interested in, namely, viscosity solutions to \eqref{eq:viscosity_sol} with graphical free boundaries. Under the mild assumption that the contact set is the subgraph of a continuous function, we show that solutions in this class are actually minimizers of the Alt--Caffarelli functional \eqref{eq:JOm}. This, in turn, allows us to use the variational structure of the problem. In particular, we consider Proposition \ref{prop:cont_graph} to be our main contribution in the classical setting and of independent interest. 

We begin by formally introducing the class of solutions with graphical free boundaries:
\begin{defi}
\label{defi.2}
Suppose that $u$ is a solution to the classical one-phase problem in $\R^n$ as in Definition \ref{DefViscositySolutionsClassical}, and that $\be\in\mathbb{S}^{n-1}$.

We say that $u$ is a \textit{graphical solution in direction $\be$}, and write
$$
u\in\mathcal{G}(\be),
$$
 if 
 $$\CSetC+\tau\be\supset\CSetC \text{ for all $\tau>0$.}$$ 
\end{defi}
Recall that the contact set $\CSetC$ is defined in \eqref{ContactSetClassical}.

\begin{rem}
This definition gives a very weak notion of graphical free boundaries. Indeed, it says that we can see the free boundary $\partial\{u > 0\}$ as a graph of a ``generalized function" over the hyperplane $\{\be \cdot x = 0\}$; such a function does not need to be defined everywhere; we only require that the intersection of $\partial \{u > 0\}$ with each line perpendicular to the hyperplane $\{\be \cdot x = 0\}$ is connected. 
\end{rem}

By definition, if a solution is monotone in the direction $\be$, then it has graphical free boundaries. We see now that the converse is true. This will be useful in turning geometric comparison of the free boundaries into analytic comparison between the solutions.

\begin{lem}
\label{lem:monotone}
Let $u\in \mathcal{G}(\be_n)$. Then $u$ is monotone nondecreasing in the direction $\be_n$. 
\end{lem}
\begin{proof}
We assume $\Lambda(u) \neq \emptyset$, otherwise $u$ is constant. Let us argue by contradiction, and let us assume that we have (with the universal gradient bound as in Lemma \ref{lem:visclip}), 
\[
\gamma := -\inf_{\R^n\backslash\CSetC}\partial_n u>0.
\]

Consider a sequence $x_i \in \{u > 0\}$ such that 
\[
\partial_n u (x_i) \to -\gamma\quad\text{as}\quad i \to \infty
\]
 and let $y_i\in \CSetC$ be such that 
$$
r_i := |x_i-y_i|= {\rm dist}(x_i, \Lambda(u)).
$$
If we rescale the solution as
\[
w_i(x) := \frac{u(r_i x + x_i)}{r_i},
\]
then
\[
\left\{
\begin{array}{rcll}
\Delta w_i &=& 0&\quad\text{in}\quad B_1,\\
w_i &\ge& 0&\quad\text{in}\quad B_1.
\end{array}
\right.
\]
With $|\nabla w_i|\le  C$ in $\R^n$ by Lemma \ref{lem:visclip}, and 
\begin{equation}
\label{TheRescaledPoints}
w_i(\overline{y}_i)=0
\quad\text{ where $\overline{y}_i:=\frac{y_i-x_i}{r_i}\in\mathbb{S}^{n-1}$,}
\end{equation}
we have, up to a subsequence, 
\[
w_i\to \overline{w}\qquad\text{locally uniformly in $\R^n$, and in $C^1_{\rm loc}(B_1)$}
\]
for some harmonic function $\overline{w}$. In particular, we have
\begin{equation}
\label{TheFunctionIsNonnegative}
\overline{w}\ge0 \text{ and } |\nabla\overline{w}|\le C \text{ in }\R^n,
\end{equation} 
and 
$$\partial_n \overline{w}(0) = -\gamma\text{ and }\partial_n\overline{w}\ge -\gamma \text{ in $B_1$.}$$

Strong maximum principle, applied to $\partial_n\overline{w}$, implies that 
$\partial_n \overline{w} \equiv -\gamma \text{ in $B_1$,}$ and we can write 
$$
\overline{w}(x',x_n) = -\gamma x_n + g(x') \text{ in $B_1$}
$$ 
for some Lipschitz function $g$ depending on  $x'\in \R^{n-1}$. 
Moreover, with $\overline{w}\ge 0$, we have $g(x') \ge \gamma |x_n|\ge 0$ for any $(x', x_n)\in B_1$. Restricting to $\partial B_1$, we have
\begin{equation}
\label{eq:contrg}
g(x') \ge \gamma \sqrt{1-|x'|^2}\ge 0\quad\text{for any}\quad |x'|\le 1. 
\end{equation}

Up to a subsequence, the points $\overline{y}_i$ from \eqref{TheRescaledPoints} converge  to some $\overline{y}\in\mathbb{S}^{n-1}$. The condition that each $w_i\in \mathcal{G}(\be_n)$ implies that $\bar{y}_i\cdot \be_n \leq 0$ and so  $
\overline{y}\cdot\be_n\le 0.
$
If $\overline{y}\cdot\be_n<0$, then $\overline{w}(\overline{y})=0$ and $\partial_n\overline{w}=-\gamma<0$ implies $\overline{w}(\overline{y}+t\be_n)<0$ for small $t>0$, contradicting \eqref{TheFunctionIsNonnegative}. Therefore,   we have
$$
\overline{y}\cdot\be_n= 0 \text{ and }|\overline{y}'|=1.
$$
As a result, we have $g(\overline{y}')=\overline{w}(\overline{y})=0$. Now we take $p\in\mathbb{S}^{n-1}$, then \eqref{eq:contrg} implies that 
$$
g(p') - g(\overline{y}') =g(p')\geq \gamma\sqrt{1-|p'|^2} \geq \frac12\gamma\sqrt{|\overline{y}'-p'|},
$$ 
contradicting the Lipschitz regularity of $g$ for $p'$ close to $\overline{y}'.$
\end{proof}

A useful corollary is the stability of the class $\mathcal G(\be)$:
\begin{cor}
\label{cor:closedgraph}
 If $u_i \in \mathcal G(\be)$ and $u_i \rightarrow u_\infty$ locally uniformly, then $u_\infty \in \mathcal G(e)$. 

\end{cor}



The following proposition establishes the variational structure behind monotone viscosity solutions. 
For this proposition, it is more convenient to use the cylindrical coordinates. For $R,L>0$, we denote by 
$$
B_R':=\{x=(x',x_n)\in\R^n:x_n=0, |x'|<R\}.
$$

\begin{prop}
\label{prop:cont_graph}
For $L>H>0$, let $u$ be a viscosity solution to the classical one-phase problem \eqref{eq:viscosity_sol} in $\Omega =B_2'\times(-2L-H,2L+H)$ with 
\[
\partial_n u\ge0 \quad \text{in}\quad \Omega.
\]

If its contact set is a subgraph
$$
\CSetC=\{(x',x_n):x_n\le f(x')\}
$$
for a continuous function $f$ satisfying
\[
-H<f<H  \quad \text{in} \quad B_2',
\]
then $u$ is the unique minimizer of the Alt--Caffarelli functional \eqref{eq:JOm} in $D=B_1'\times(-L,L)$. 
\end{prop}

\begin{rem}
With Lemma \ref{lem:monotone}, Proposition \ref{prop:cont_graph} implies Theorem \ref{MinimizingClassicalIntro}. 
\end{rem}
\begin{rem}
This is the only reason why we require the free boundary to be continuous in the main results. 
\end{rem} 

Proposition \ref{prop:cont_graph} follows from the following two lemmata, where we show, respectively, that $u$ is no less than any minimizer, and that $u$ is no larger than any minimizer in $D$. 
 
\begin{lem}
\label{lem:cont_graph1}
Under the same assumptions as in Proposition \ref{prop:cont_graph},  let $w$ be a minimizer of the Alt--Caffarelli functional \eqref{eq:JOm} in $D$ with $w=u$ on $\partial D$. 

Then $u\ge w$ in $D$. 
\end{lem} 

\begin{proof}
Suppose not;  then, there exist  some $x_\circ\in D$ and $\eta_\circ > 0$ such that
\[
w(x_\circ) > u(x_\circ) + \eta_\circ. 
\]
For $\tau \in \R$, define the translation of $u$ as 
\[
u_\tau(x', x_n) := u(x', x_n + \tau).
\]
Fix $s > 0$ small such that $$w(x_\circ) > u_s(x_\circ) + \frac12 \eta_\circ.$$

\vem 

\noindent \textit{Step 1: Setting up the inf-convolution.} 

By monotonicity of $u$ and the uniform continuity of the free boundary in $D$, there is a set  $E$ such that
\[
\{u > 0\}\cap D \Subset E\Subset \{u_{s} > 0\},
\]
(see Figure~\ref{fig:1}). By strict maximum principle in the interior of $E$, we have
\[
\inf_E \partial_n u_s > 0,
\]
which gives $\delta  >0$ such that 
\begin{equation}
\label{eq:from3}
u_s \ge u + \delta \quad\text{in}\quad \overline{\left\{u > 0\right\}\cap D}.
\end{equation}

\begin{figure}
\centering
\includegraphics[scale = 0.7]{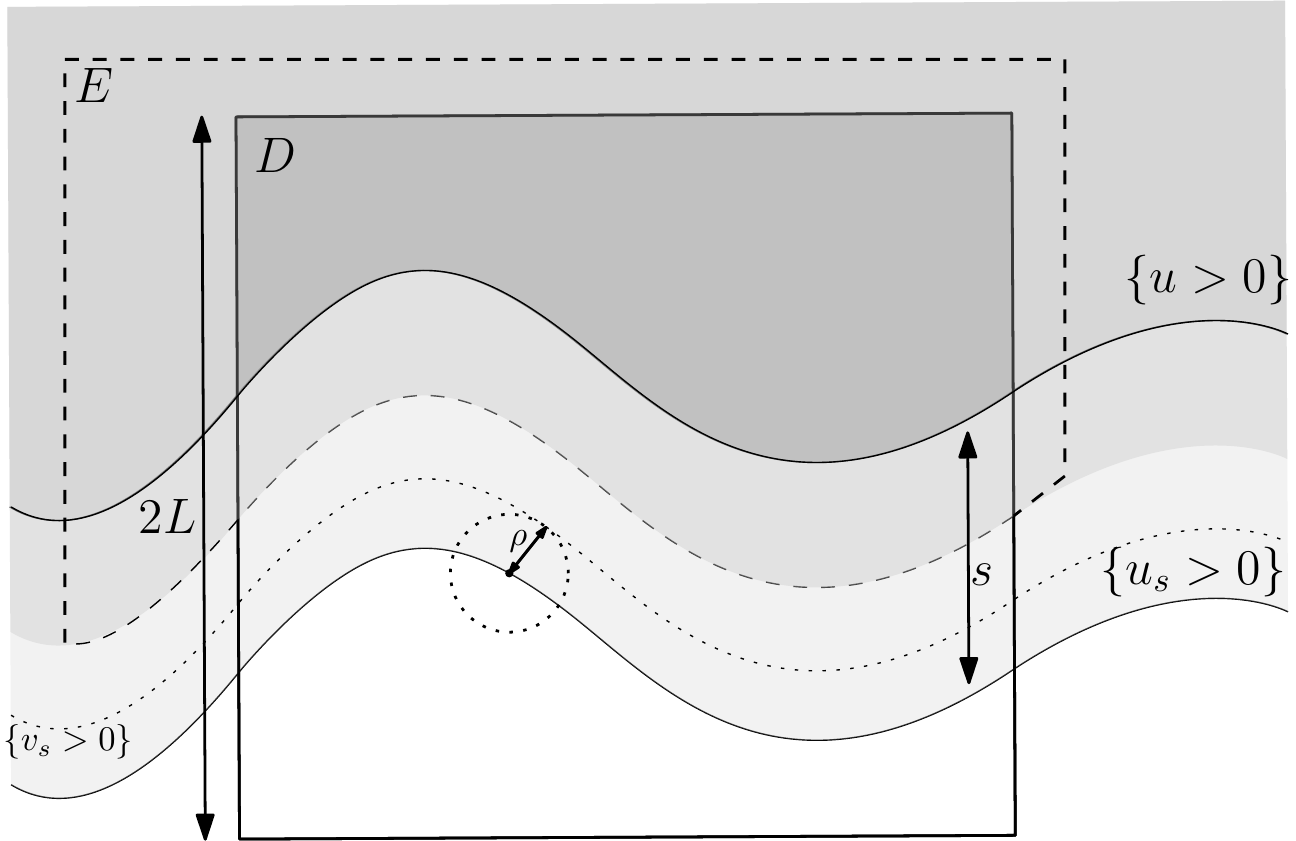}
\caption{Setting in the proof of Lemma~\ref{lem:cont_graph1}}.
\label{fig:1}
\end{figure}
For $\rho > 0$ small denote the inf-convolution of $u_\tau$, as in Definition \ref{DefConvolutionsClassical}, 
\[
v_\tau(x) := \inf_{B_\rho(x)} u_\tau.
\]
By the monotonicity of $u$, we have 
$$v_{s+t} \geq v_s, \qquad \forall t \geq 0.$$ 
Moreover, if we pick $\rho > 0$ small enough (depending on $\delta$ and  the modulus of continuity for $u$), we have (in light of \eqref{eq:from3}) that
\begin{equation}
\label{eq:fromfrom}
v_s(x) \ge u(x)  \quad\text{for all}~~x\in \overline{D},\quad\text{and}\quad w(x_\circ) > v_s(x_\circ) + \frac14 \eta_\circ.
\end{equation}

\vem

\noindent \textit{Step 2: Initializing the sliding argument.}

By the upper bound on $f$ as in Proposition \ref{prop:cont_graph}, we see that if $t$ is large enough such that $t>\rho+H+L-s$, then 
$ v_{s+t} > 0$ in $\overline{D}.$ 
With Lemma \ref{PropertyOfConvolutionsClassical}, this implies
\[
\Delta v_{s+t} \le 0\quad\text{in}\quad D.
\]
On the other hand, we know that for all $t \ge 0$, 
$v_{s+t} \ge  u= w$ on $\partial D$. 
Since $\Delta w \ge 0$ in $D$, we have
\begin{equation}
\label{eq:from44}
v_{s+t} \ge w \quad\text{in}\quad \overline{D}
\end{equation}
if $t>\rho+H+L-s$.

Let us define now the critical contact time 
\[
t_* = \inf\{t \ge 0 : v_{s+t}\ge w \text{ in }\overline{D}\}.
\]
From \eqref{eq:fromfrom}, $t_* > 0$. 

\vem

\noindent \textit{Step 3: The contact point in the sliding argument.}

Let $\overline{x} \in \overline{\{w > 0\} \cap D}$ be such that 
$$v_{s+t^*}(\overline{x}) = w(\overline{x}).$$ Note that such a touching point must exist, otherwise the nonnegativity and monotonicity of $u$ would imply that $v_{s+t^*-\varepsilon} \geq w$ for some small $\varepsilon$, contradicting the definition of $t^*$. 

With \eqref{eq:from3}, if we take $\rho$ small, then we can assume
\begin{equation}
\label{eq:from4}
v_{s+t^*}\ge v_s \ge u + \frac{\delta}{2} = w + \frac{\delta}{2} \quad\text{on}\quad \overline{\left\{u > 0\right\}}\cap \partial D.
\end{equation}
Thus 
$\overline{x} \notin \{w > 0\} \cap \partial D.$ Meanwhille, in  $\{w > 0\}\cap D$, we have $\Delta v_{s+t^*} \leq 0 = \Delta w$. Combined with $v_{s+t^*} \geq w$, this tells us that $v_{s+t^*} > w$ in $\{w > 0\}\cap D$ and that 
$\overline{x} \notin \{w > 0\} \cap D.$
As a result, we must have  
$$v_{s+t_*}(\bar x) = w(\bar x) = 0.$$ 

\vem

\noindent \textit{Step 4: The contradiction.}

There are two possibilities to consider,  depending on whether  this touching point lies on $\partial D$ or inside $D$.

 If $\bar x\in \partial \{w > 0\}\cap D$, then  we have $\bar x \in \partial\{v_{s+t_*} > 0\}$. With the existence of a tangent ball as in Lemma \ref{PropertyOfConvolutionsClassical}, the point $\bar{x}$ is a regular point of $\partial \{w > 0\}$ (see, e.g. \cite[Theorem 8.1]{AC81}). 


Since $w$ is a minimizer, we have 
\[
w(x) = \langle x - \bar x,   \nu\rangle_+ + o(|x-\bar x|)
\]
where $\nu$ is the inner unit normal of $\partial\PosS$ at $\overline{x}$.
On the other hand, the supersolution property in Lemma \ref{PropertyOfConvolutionsClassical} implies
\[
v_{s+t_*}(x) \le\langle x-\bar x, \nu\rangle_+ + o(|x-\bar x|).
\]
These contradict Hopf's lemma for the nonnegative harmonic function $v_{s+t^*}-w$ at $\overline{x}.$ Consequently, we must have
$$\bar x\in \partial(\{w > 0\}\cap D)\cap \{w = 0\}\cap \partial D.$$


  With $v_{s+t^*}(\overline{x})=0$ and  \eqref{eq:from4}, we have $\bar x\notin\partial\{u > 0\}\cap \partial D$ and thus, there is a neighborhood $Z\subset\partial D$ of $\bar x $ where 
  $$
  w=u = 0\text{ on $Z$.}
  $$
   In particular, we are in the situation of Theorem~\ref{thm:CS}, which means that  $\partial(\{w > 0\}\cap D)\cap \{w = 0\}$ is $C^1$ around $\bar x$, and $\nabla w$ is well-defined at $\bar x$ (since the normal is well-defined) with $|\nabla w(\bar x)|\ge 1$. Proceeding as in the previous setting, we get again a contradiction with Hopf's Lemma at $\bar x$.
   \end{proof}

\begin{lem}
\label{lem:cont_graph2}
Under the same assumptions as in Proposition \ref{prop:cont_graph},  let $w$ be a minimizer of the Alt--Caffarelli functional \eqref{eq:JOm} in $D$ with $w=u$ on $\partial D$. 

Then $u\le w$ in $D$. 
\end{lem} 

\begin{proof}
Suppose not; we find $x_\circ\in D$ and $\eta_\circ$ such that 
\[
u(x_\circ) > w(x_\circ) + \eta_\circ. 
\]
With the same notation for the translation as in the previous proof, we fix
$s > 0$ small such that $u_{-s}(x_\circ) > w(x_\circ) + \frac12 \eta_\circ$. 

As before, there exists some $\delta  >0$ small such that 
\begin{equation}
\label{eq:from5}
u \ge u_{-s} + \delta \quad\text{in}\quad \overline{\left\{u_{-s} > 0\right\}\cap D},
\end{equation}
and some $\rho$ small enough such that 
\begin{equation}
\label{eq:fromfrom2}
\tilde v_{-s}(x) := \sup_{B_\rho(x)} u_{-s} \le u(x) \quad\text{for all}~~x\in  \overline{ D},\quad\text{and}\quad \tilde v_{-s}(x_\circ) > w(x_\circ) + \frac14 \eta_\circ.
\end{equation}

With the assumption on the lower bound on $f$ as in Proposition \ref{prop:cont_graph}, we have 
\[
\tilde v_{-s-t} \equiv 0 \le w \quad\text{in}\quad \overline{D}
\]
if $t+s>L+H+\rho$.
Also, from \eqref{eq:from5} (taking $\rho$ smaller if necessary)
\begin{equation}
\label{eq:from6}
\tilde v_{-s-t} \le w - \frac{\delta}{2} \quad\text{on}\quad \overline{\left\{\tilde v_{-s} > 0\right\}\cap \partial D}\supset \overline{\left\{\tilde v_{-s-t} > 0\right\}\cap \partial D}.
\end{equation}

\vem 

We define
\[
t_* = \inf\{t \ge 0 : \tilde v_{-s-t}\le w \text{ in }\overline{D}\}.
\]
Arguing as before, we have $t_* > 0$, and there exists some $\bar x \in \overline{\{\tilde v_{-s-t_*} > 0\}\cap D}$ such that 
$$
\tilde v_{-s-t_*}(\bar x) = w(\bar x). 
$$
Moreover, we have 
 $\bar x\notin \{\tilde v_{-s-t_*} > 0\}\cap D$ by the maximum principle, and 
 $\bar x\notin  \partial D$ by \eqref{eq:from6}.

As a result, we have
$$
\bar x\in \partial\{\tilde v_{-s-t_*}>0\}\cap D.
$$
With the asymptotic expansion of $w$ and $v_{-s-t_*}$ from Definition \ref{DefViscositySolutionsClassical} and Lemma \ref{PropertyOfConvolutionsClassical}, this again contradicts Hopf's Lemma.
\end{proof}

Thus, as a consequence of the previous two lemmata, we obtain:
\begin{proof}[Proof of Proposition~\ref{prop:cont_graph}]
It is a combination of Lemmas~\ref{lem:cont_graph1} and \ref{lem:cont_graph2}. 
\end{proof}

We finally have: 

\begin{proof}[Proof of Theorem~\ref{MinimizingClassicalIntro}]
It follows from Lemma \ref{lem:monotone} and Proposition \ref{prop:cont_graph}
\end{proof}

\section{Flatness of graphical solutions: the classical regime}
\label{sec:4}

In this section we prove our main result in the classical regime, namely, Theorem \ref{thm:main0}. With Theorem \ref{MinimizingClassicalIntro} (see also Lemma \ref{lem:monotone} and Proposition \ref{prop:cont_graph}), it suffices to consider global minimizers. 

We start with the following technical lemma, which says that if $u$ is monotone in the direction $\be$ and has smooth free boundary, then either $\be$ is never tangent to the free boundary or the solution is independent of the direction $\be$. 

\begin{lem}
\label{lem:interm_lem}
Let $u$ be a viscosity solution in the sense to \eqref{eq:viscosity_sol} in $B_1$ with $\partial\CSetC\cap B_1$ being $C^{2}$-submanifold with inward pointing unit normal  $\nu$. Also assume that $\{u > 0\}\cap B_1$ is connected. 

If, for some $\be\in\mathbb{S}^{n-1}$, $$\partial_\be u \ge 0 \text{ in $B_1$},$$ 
then, either $\nu(x) \cdot \be > 0$ for all $x\in  \partial\{u > 0\}\cap B_1$, or $\partial_\be u \equiv 0$ in $B_1$.  
\end{lem}

\begin{proof}
Suppose not; we have $$\partial_\be u \not\equiv 0\text{ in $B_1$,}$$ but 
$$\nu(x_\circ) \cdot\be = 0\text{ for some $x_\circ \in \partial\{u > 0\}\cap B_1$.}$$ As such $\partial_{\be} u (x_\circ) =  \be\cdot \nu(x_\circ) = 0$. 

Since $\partial_\be u \not\equiv 0$ and $\partial_\be u \ge 0$,  we can apply Hopf's Lemma  to deduce that 
\[
\partial_{\be} \partial_{\nu(x_\circ)} u(x_\circ) = \partial_{\nu(x_\circ)} \partial_{\be}u(x_\circ)  > 0. 
\]
On the other hand, the function $\partial\{u > 0\} \ni x \mapsto \partial_{\nu(x_\circ)} u(x)$ has a maximum at $x_0$. As $\be$ is tangent to $\partial \{u > 0\}$ at $x_0$ we get $\partial_{\be} \partial_{\nu(x_\circ)} u(x_\circ) = 0$, the desired contradiction. 
\end{proof}

With this lemma, we show that graphical cones are flat in low dimensions. Recall the critical dimension $n_{\rm local}^*$ defined in \eqref{LargestDimensionClassical} and the notion of global minimizers from Definition \ref{DefMinimizersClassical}.
\begin{prop}
\label{prop:main}
Let $u\in \mathcal{G}(\be_n)$ in $\R^n$ with 
$$
n\le n_{\rm local}^*+1.
$$ 
If $u$ is a homogeneous minimizer,  then 
$$u = (x\cdot \be)_+$$
 for some
  $\be\cdot\be_n \ge 0$. 
\end{prop}





\begin{proof}
Lemma \ref{TraceSmoothClassical} implies that for each $x\in \partial \{u > 0\}\cap \mathbb S^{n-1}$, the unit normal $\nu(x)$ to $\partial\{u > 0\}$ (outward with respect to $\{u = 0\}$) exists and is a continuous function of $x$. The assumption $u\in \mathcal G(\be_n)$ implies that $\be_n \cdot \nu(p) \geq 0$ for all $p \in \partial \{u > 0\}\cap \mathbb S^{n-1}$. By continuity there exists a direction
\[
\be_\circ\in \argmin_{\bar\be \in \mathbb{S}^{n-1}} \left\{\bar\be \cdot\be_n : \bar\be\cdot \nu(x) \geq 0, \ \forall x\in \partial \{u > 0\}\cap \mathbb S^{n-1}\right\}.
\]

We claim that there is a point $p_\circ\in \partial \{u > 0\}\cap \mathbb S^{n-1}$ such that $\be_\circ \cdot \nu(p_\circ) = 0$. If not, then by compactness there is a $\delta > 0$ such that $\be_\circ \cdot \nu(p) \geq \delta$ for all $p \in \partial \{u > 0\}\cap \mathbb S^{n-1}$. This implies that for any $\bar\be \in \mathbb S^{n-1}$ with $\|\bar\be - \be_\circ\| < \delta/2$ we have $\bar\be \cdot \nu(p) \geq \delta/2>0$ for all $p \in \partial \{u > 0\} \cap \mathbb S^{n-1}$, contradicting the minimality of $\be_\circ$.

If $\be_n = \be_\circ$, let $p_\circ\in \partial \{u > 0\} \cap \mathbb S^{n-1}$ be such that $\nu(p_\circ)\cdot \be_n = 0$. Recall that for every globally defined minimizer $u$, $\{u > 0\}$ is  connected (see, e.g. \cite[Theorem 2.2]{DET} or \cite[Theorem 2.3]{ESV22}). Hence, we can apply Lemma~\ref{lem:interm_lem} to the connected component of $B_{1/2}(p_\circ)\cap \{u > 0\}$ with $p_\circ$ on its boundary ($u$ is monotone in the $e_n$ direction, by Lemma~\ref{lem:monotone}), to conclude that $u$ is invariant in the direction $\be_n$ in all of $\R^n$ (by analyticity and connectedness of $\{ u > 0\}$). As a result, the restriction of $u$ into the space perpindicular to $\be_n$ is a minimizing cone in $\mathbb R^{n-1}$. The criticality of $n_{\rm local}^*$ implies that $u$ is a half-plane solution. 

So we may assume $\be_n\cdot \nu(p) > 0$ for all $p \in \partial \{u> 0\}\cap \mathbb S^{n-1}$ and thus $\be_n \neq \be_\circ$.  If we can show that $u\in \mathcal G(\be_\circ)$,   we may argue as above around the point $p_\circ$ (where $ \be_\circ \cdot \nu(p_\circ)= 0$) to conclude that $u$ is a half-plane solution. In order to prove that $u\in \mathcal G(\be_\circ)$, we first note that because $\be_n\cdot \nu(p) > 0$, by homogeneity, and by Lemma \ref{TraceSmoothClassical}, we have that $\partial \{u > 0\}$ is the graph of a Lipschitz function in the $\be_n$ direction (and in fact, $\partial \{u > 0\}\setminus\{0\}$ is a smooth graph). A simple computation shows that for any $\delta > 0$, we have  $(\be_\circ + \delta \be_n)\cdot\nu(x) > 0$ for all $x\in \partial \{u > 0\} \cap \mathbb S^{n-1}$. So by the implicit function theorem, $\partial \{u > 0\}\cap \mathbb S^{n-1}$ is the graph of a smooth function over the equator perpendicular to $\frac{\be_\circ + \delta \be_n}{\|\be_\circ + \delta \be_n\|}$. By homogeneity this implies that $u\in \mathcal G(\frac{\be_\circ + \delta \be_n}{\|\be_\circ + \delta \be_n\|} )$ for all $\delta > 0$. Sending $\delta \downarrow 0$ and invoking Corollary \ref{cor:closedgraph} we are done. 
\end{proof}
 
We then have, by a blow-down argument:
\begin{cor}
\label{cor:main000}
Let $u\in \mathcal{G}(\be_n)$ be a global minimizer to the Alt--Caffarelli functional in $\R^n$ with 
$$n\le n_{\rm local}^*+1.$$ Then $u = (x\cdot \be)_+$  for some  $\be\cdot\be_n \ge 0$. 
\end{cor}
\begin{proof}
Consider the rescalings
\[
u_R(x) = \frac{u(Rx)}{R}
\]
as $R\to \infty$. By Lemma \ref{lem:blowdown}, we have
\[
u_{R_i} \to v\qquad\text{locally uniformly}
\]
along a subsequence $R_i\uparrow \infty$, where $v$ is some homogeneous minimizer to the one-phase problem.  With Corollary \ref{cor:closedgraph} and  Proposition~\ref{prop:main}, we have 
$$
v = (x\cdot \be')_+
$$ 
for some $\be' \in \mathbb{S}^{n-1}$.

Given small $\eps>0$, we have $$\|u_{R_i}-v\|_{L^\infty(B_1)}<\eps\qquad \text{for $i$ large enough}.$$ From here, we iterate Lemma \ref{LemIOFClassical} to conclude
$
|u_{R_i}-v_k|\le(\frac{r_0}{2})^{k}\eps \text{ in $B_{r_0^k}$}
$
where each $v_k$ is a half-plane solution. That is, 
$
|u-v_k|\le R_i(\frac{r_0}{2})^{k}\eps \text{ in $B_{R_ir_0^k}$}.
$
Choosing $R_i$ and $k$ large enough, we conclude 
$$
\|u-v_k\|_{L^\infty(B_1)}\le\eps.
$$
Since $\eps$ is arbitrary and the set of half-plane solutions compact, we conclude $u$ is a half-plane solution in $B_1$. A similar argument can be used to show that $u$ is a half-plane solution in any compact subset of $\R^n$. That $\be \cdot \be_n \geq 0$ follows immediately the fact that $u \in \mathcal G(\be_n)$. 
\end{proof}

Combining the previous results we directly get Theorem \ref{thm:main0}:

\begin{proof}[Proof of Theorem \ref{thm:main0}]
Thanks to Theorem~\ref{MinimizingClassicalIntro}, $u$ is a global minimizer. We are now done by Corollary~\ref{cor:main000}. 
\end{proof}
 
And we also get Corollary~\ref{cor:semilinear}:
\begin{proof}[Proof of Corollary~\ref{cor:semilinear}]
Suppose that $u$ is a solution to $\Delta u=f(u)$. The condition \eqref{eq:cond_min} and  $\partial_{x_n} u > 0$ implies that $u$ is a minimizer of the corresponding energy functional. This can be proven by constructing a foliation. In fact, the same proof used in \cite[Theorem 2.4]{CP18} works in this context, where the condition \eqref{eq:cond_min} ensures that (large) translations of $u$ are completely above or below a potential minimizing competitor on a given compact set $K$ (see also the proof of \cite[Theorem 4.4]{AAC01}).

The result  is now a consequence of \cite{AS22}. We use
\cite[Proposition 5.1]{AS22}  to obtain that an appropriate rescaling is arbitrarily close to a global solution to the one-phase problem. Since the graphicality condition in Definition~\ref{defi.2} passes well to the limit (see also \cite[Lemma~5.2]{AS22}), thanks to our classification result in Corollary~\ref{cor:main000} we are done by applying \cite[Theorem 1.4]{AS22}.
\end{proof}

\section{Preliminaries and notations: the thin case}\label{sec:thin}
\label{sec:5}

In this section, we collect some preliminary facts about solutions to the thin one-phase problem \eqref{eq:visc_sol_s}. 

We begin with the definition of viscosity solutions to \eqref{eq:visc_sol_s}, see, e.g. \cite{DR12} or \cite{DS15b}, which parallels the classical definition (recall Definitions~\ref{defi:comp_sol} and~\ref{DefViscositySolutionsClassical}). 

In the following, we denote by $F(u)$ the free boundary of $u\ge 0$ in $\Omega$, which is the boundary of a set in $\{x_{n+1} = 0\}$ (with respect to its relative topology):
 \[
F(u) := \partial_{\R^n} \left(\{u > 0\}\cap \{x_{n+1} = 0\}\right)\cap \Omega,
 \]
 and we also denote 
 \[
 \mathbb{S}^n_0 :=\{\be \in \mathbb{S}^n : \be_{n+1} = 0\} = \{\be = (\be', \be_{n+1}) \in \R^{n}\times \R : |\be'| = 1,\ \be_{n+1} = 0\}. 
 \]
Finally, recall from \eqref{HalfSpaceSolutionThin} the one-phase solution: $$U(x_n, y) := \frac{1}{\sqrt2}\sqrt{x_n+\sqrt{x_n^2+y^2}}.$$

\begin{defi}\label{DefViscositySolutionThin1}
Let $u\in C(\Omega)$ for some domain $\Omega\subset \R^{n+1}$, $u \ge 0$ in $\Omega$, even with respect to the plane $\{x_{n+1} = 0\}$: 
\begin{enumerate}[leftmargin=*,label=(\roman*)]
\item We say that $u$ is a (strict) \textit{comparison subsolution} to the thin one-phase problem \eqref{eq:visc_sol_s} if  
 \[
u\in C^2(\PosS),\qquad \Delta u \ge 0\quad\text{ in }\quad \PosS,\]
the free boundary  $F(u)$  is a $C^2$ mani\-fold, and for any $x_\circ \in F(u)$ there exists a $\alpha(x_\circ) > 1$ such that, denoting $z = (z', z_{n+1})\in \R^n\times \R$, 
  $$
 u(x_\circ + z) =\alpha(x_\circ)U(z'\cdot \nu, z_{n+1}) +o(|z|^{1/2}),$$
where $\nu\in \mathbb{S}^n_0$ is the inward normal to $F(u)$ at $x_\circ$, and $U(x_n, y)$ is given by \eqref{HalfSpaceSolutionThin}.  

\item We say that $u$ is a (strict) \textit{comparison supersolution} to the thin one-phase problem \eqref{eq:visc_sol_s} if  
 \[
u\in C^2(\PosS),\qquad \Delta u \le 0\quad\text{ in }\quad \PosS,\]
the free boundary  $F(u)$  is a $C^2$ mani\-fold, and for any $x_\circ \in F(u)$ there exists a $\alpha(x_\circ) < 1$ such that, denoting $z = (z', z_{n+1})\in \R^n\times \R$, 
  $$
 u(x_\circ + z) =\alpha(x_\circ)U(z'\cdot \nu, z_{n+1}) +o(|z|^{1/2}),$$
where $\nu\in \mathbb{S}^n_0$ is the inward normal to $F(u)$ at $x_\circ$, and $U(x_n, y)$ is given by \eqref{HalfSpaceSolutionThin}.  
%
\end{enumerate}
\end{defi}

As in the classical case, we use these comparison solutions as test functions to define a \emph{viscosity solution}:

\begin{defi}
\label{DefViscositySolutionThin}
Let $u\in C(\Omega)$ for some domain $\Omega\subset \R^{n+1}$, $u \ge 0$ in $\Omega$, even with respect to the plane $\{x_{n+1} = 0\}$. We say that $u$ is a \emph{viscosity solution} to the thin one-phase problem \eqref{eq:visc_sol_s} if  
$$
\Delta u = 0\quad \text{ in }\quad \{u>0\}\cap\Omega,
$$ 
and any strict comparison subsolution (resp. supersolution) cannot touch $u$ from below (resp. from above) at a free boundary point $x_\circ\in F(u)$. 
\end{defi} 

In the previous definition, we say that a strict comparison subsolution $v$ touches from below $u$ at a free boundary point $x_\circ\in F(u)$ if $x_\circ\in F(v)$ and $v \le u$ in a neighborhood of $x_\circ$.

As in the classical case we want to define the sup/inf-convolutions. Note that in this setting   the neighborhoods over which we are taking the supremum and infimum are ``thin", 

\begin{defi}\label{Definitionthinsupconv}
For a domain $\Omega \subset \mathbb R^{n+1}$ (even with respect to $\{x_{n+1} = 0\}$) and $t > 0$, define $$\Omega_{t,\mathrm{thin}} := \{(x',y)\in \Omega : (x', y+\tau) \in \Omega\quad\text{for all}\quad \tau \in (-t, t)\}.$$ For $u \in C(\Omega)$, its \emph{$t$-sup-convolution} is defined in $\Omega_{t, {\rm thin}}$ as $$\overline{u}_t(x,y):= \sup_{\{(x', y): |x-x'| < t\}} u(x',y).$$
Its \emph{$t$-inf-convolution} is defined on $\Omega_{t, {\rm thin}}$ as $$\underline{u}_t(x,y):=\inf_{\{(x', y):  |x-x'| < t\}} u(x',y).$$
\end{defi}

As in the classical case, these convolutions satisfy good comparison properties (the proof of this lemma follows as in the classical case once one has  \cite[Lemma 7.5]{DS12}, see also \cite[Corollary 2.9]{DSS14}).

\begin{lem}
\label{PropertyOfConvolutionsThin}
Let $u\in C(\Omega)$ be a viscosity solution to the thin one-phase problem \eqref{eq:visc_sol_s} in $\Omega$. For $t>0$, let $\overline{u}_t$ and $\underline{u}_t$ denote its sup-convolution and inf-convolution as in Definition \ref{Definitionthinsupconv}. Then:

\begin{itemize}
\item $\overline{u}_t$ satisfies $\Delta\overline{u}_t \ge 0$ in $\{\overline{u}_t > 0\}\cap \Omega_{t, {\rm thin}}$ and, for each $x_\circ\in F(\overline{u}_t)$, there is a point $p = (p', 0)$ such that 
\[
B'_t(p) \subset \{\overline{u}_t > 0\}\cap \{x_{n+1} = 0\} \quad\text{and}\quad x_\circ \in \partial B_t(p),
\]
and 
\[
\overline{u}_t(x', 0) \ge \langle x'-x_\circ', \nu\rangle^{1/2}_+ + o(|x'-x'_\circ|^{1/2})
\]
for $x'$ near $x_\circ'$, where $\nu :=\frac{1}{t}(p'-x'_\circ)$.
\item $\underline{u}_t$ satisfies $\Delta\underline{u}_t \le 0$ in $\{\underline{u}_t > 0\}\cap \Omega_{t, {\rm thin}}$ and, for each $x_\circ\in F(\underline{u}_t)$, there is a point $p = (p', 0)$ such that 
\[
B'_t(p) \subset \{\underline{u}_t = 0\} \cap \{x_{n+1} = 0\} \quad\text{and}\quad x_\circ \in \partial B_t(p),
\]
and 
\[
\underline{u}_t(x', 0) \le \langle x'-x'_\circ, \nu\rangle^{1/2}_+ + o(|x'-x'_\circ|^{1/2})
\]
for $x'$ near $x_\circ'$, where $\nu :=\frac{1}{t}(x'_\circ-p')$.
\end{itemize}
\end{lem} 

We turn now to the regularity of viscosity solutions. Corresponding to Lemma \ref{lem:visclip}, we have the following (with an analogous proof):

\begin{lem}
\label{lem:viscC1/2}
Let $u$ be a solution in $\R^{n+1}$ to the thin one-phase problem \ref{eq:visc_sol_s}. Then, there is a dimensional constant $C$ such that 
$$
[ u]_{C^{1/2}(\R^{n+1})}\le C \text{ in $\R^{n+1}$.} 
$$
\end{lem}

In terms of the free boundary, we have an improvement of flatness lemma \cite[Theorem 7.1]{DR12}:

\begin{lem}\label{l:improvethinflat}
Let $u$ be a viscosity solution to the thin one-phase problem \eqref{eq:visc_sol_s} in $B_1$, and assume that $0 \in F(u)$ and   $$U(x_n-\eps, x_{n+1}) \leq u(x) \leq U(x_n + \eps, x_{n+1}), \quad\text{for all}\quad x\in B_1,$$
where $U(x_n, x_{n+1})$ is given by \eqref{HalfSpaceSolutionThin}. 

There are dimensional constants  $\bar {\eps} > 0$ and $r > 0$ such that if $\eps  \leq \bar {\eps}$, then  we can find $\be\in \mathbb{S}^{n-1}$ such that 
\[
U(x\cdot \be -\eps r / 2,x_{n+1}) \leq u(x) \leq U(x\cdot \be + \eps r /2 , x_{n+1}), \quad\text{for all}\quad x \in B_r.
\]
\end{lem}

A particular class of solutions to the thin problem are minimizers of an appropriate energy functional, \eqref{eq:JOm_frac}:

 \begin{defi}
 \label{DefMinimizersthin}
 For $\Omega\subset\R^{n+1}$ and $u\in H^1(\Omega)$, both even with respect to $\{x_{n+1} = 0\}$, we say that $u$ is a \emph{minimizer} of the thin Alt--Caffarelli functional \eqref{eq:JOm_frac} in $\Omega$ if 
$u\ge0$ in $\Omega$, and 
\[
\mathcal{J}^0_\Omega(u) \le \mathcal{J}^0_\Omega(v)\qquad\text{for all}\quad v\ge0, \quad v-u\in H^1_0(\Omega). 
\]

For $u\in H^1_{\rm loc}(\R^{n+1})$ with $u\ge 0$ and even with respect to $\{x_{n+1} = 0\}$, we say that it is a \textit{global minimizer} in $\R^{n+1}$ if
it is a minimizer in $B_R$ for every $R>0$.
\end{defi}

As in the classical setting, minimizers have nondegeneracy and compactness properties that allow for additional arguments. In particular, we can execute a blow-down argument using a Weiss-type monotonicity formula (see, e.g. \cite{All}):
 
\begin{lem}\label{lem:nonlocalblowdown}
Let $u$ be a global minimizer of the thin Alt--Caffarelli functional in $\mathbb R^{n+1}$.
For a sequence $r_i \uparrow \infty$, define 
\[
u_i(x) := \frac{u(r_ix)}{r_i^{1/2}}.
\]
 Then, perhaps passing to a subsequence, we can find a nonzero $1/2$-homogeneous global minimizer $u_\infty$ such that $$u_i \rightarrow u_\infty\qquad \text{locally uniformly in}\;\; \mathbb R^{n+1}$$ with $$\chi_{\{u_i = 0\}\cap \{x_{n+1}=0\}} \rightarrow \chi_{\{u_\infty = 0\}\cap \{x_{n+1}=0\}}\;\;\text{in}\;\; L^1_{\mathrm{loc}}(\{x_{n+1} = 0\}),$$ and 
 $$F(u_i) \rightarrow F(u_\infty)\; \text{ locally in the Hausdorff distance sense.}$$
\end{lem}
\begin{proof}
The proof is the same as the local setting (Lemma \ref{lem:blowdown}) using the Weiss-type monotonicity formula adapted to the thin case in \cite{All} and the compactness properties of minimizers to the thin functional (see, e.g. \cite[Lemma 3.4]{EKPSS20}).
\end{proof}

Finally,   in analogy to the classical setting, homogeneous minimizers have smooth free boundaries on the sphere in low dimensions.  Recall the critical dimension $n_{\rm thin}^*$ defined in \eqref{LargestDimensionThin}.
\begin{lem}
\label{TraceSmoothThin}
Suppose that $u$ is a homogeneous minimizer in $\R^{n+1}$ with 
$$
n\le n_{\rm thin}^*+1.
$$
Then $F(u)\cap\mathbb{S}_0^{n}$ is smooth.
\end{lem} 

\begin{rem}
For the thin one-phase problem, this is the only place where we require the restriction on dimension. 
\end{rem} 

\begin{proof}
The proof proceeds exactly as in the classical case (Lemma \ref{TraceSmoothClassical}). 
\end{proof} 

\section{Graphical solutions are minimizers: the thin case}
\label{sec:6}

As in the classical case, we first show that graphical solutions are minimizers to the thin one-phase functional. The first step is adapting the definition of graphical free boundaries, Definition \ref{defi.2}, to the thin setting:

\begin{defi}
\label{defi.3}
Let $u$ be a solution to the thin one-phase problem in $\R^{n+1}$ as in Definition \ref{DefViscositySolutionThin}, and that $\be\in\mathbb{S}_0^{n}$.

We say that $u$ is a \textit{graphical solution in direction $\be$}, and write
$$
u\in\mathcal{G}_s(\be),
$$
 if 
 $$\CSetC+\tau\be\supset\CSetC \text{ for all $\tau>0$.}$$ 
\end{defi}
Recall that the contact set $\CSetC$ is defined in \eqref{ContactSetThin}.

\vem

By definition, if a solution is monotone in the direction $\be$, then it has a graphical free boundary in that direction. As in the classical case,  the converse is also true in the thin setting.  

Actually, the statement in the thin case is more general as it does not involve the free boundary condition (see Remark~\ref{rem:alt_proof} for a direct proof that uses the free boundary condition). This is, in part, due to the fact that we can apply the boundary Harnack inequality in \emph{any} slit domain \cite{DS20}.

\begin{prop}
\label{thm:monotone_frac}
Let  $u$ be a solution to  
\[
\left\{
\begin{array}{rcll}
\Delta u & = & 0 & \quad\text{in}\quad \R^{n+1}\setminus \Lambda\\
u & \ge  & 0  &\quad\text{in}\quad \R^{n+1}\\
u & = & 0 & \quad\text{on}\quad \Lambda,
\end{array}
\right.
\]
where $\Lambda = \{u = 0\}\subset \{y = 0\}.$

If there is $\be\in\mathbb{S}_0^{n}$ such that $$\CSetC+\tau\be\supset\CSetC \text{ for all $\tau>0$,}$$ 
then $u$ is monotone nondecreasing in the direction $\be$. 
\end{prop}

\begin{proof}
Let us define, for some $\tau > 0$, 
\[
u_\tau(x) = u(x-\tau \be). 
\]
We have
\[
\left\{
\begin{array}{rcll}
\Delta u_\tau & \ge & 0 & \quad\text{in}\quad \R^{n+1}\setminus \Lambda\\
u_\tau & \ge  & 0  &\quad\text{in}\quad \R^{n+1}\\
u_\tau & = & 0 & \quad\text{on}\quad \Lambda.
\end{array}
\right.
\]
We have used here that $\Delta u \ge 0$ globally, and  $u_\tau = 0$ on~$\Lambda$. Thus, $u_\tau$ and $u$ are globally defined nonnegative and continuous functions that vanish continuously on some slit domain $\Lambda$, and $u$ is harmonic outside of $\Lambda$, whereas $u_\tau$ is subharmonic (and harmonic outside the thin space). 

Boundary Harnack inequality for slit domains \cite[Corollary 3.4]{DS20} (see also \cite[Theorem 1.8]{RT21}) for even functions gives a constant $C$ depending only on $n$ such that
\begin{equation}
\label{eq:fromhere}
g_\tau(R) u_\tau \le C u \quad\text{in}\quad B_R,
\end{equation}
where $g_\tau(R)$ denotes
\[
g_\tau(R) := \frac{u(R\be_{n+1})}{u_\tau(R\be_{n+1})}.
\]
 We comment that \cite[Corollary 3.4]{DS20} concerns two solutions whereas $u_\tau$ is a subsolution. However an inspection of the proof shows that the one sided inequality \eqref{eq:fromhere} holds when $u_\tau$ is a subsolution. 

Let us start by bounding $g_\tau(R)$ in terms of $\tau$ and $R$:
\[
|g_\tau(R)-1|= \frac{|u(R\be_{n+1})-u_\tau(R\be_{n+1})|}{u_\tau(R\be_{n+1})}\le \tau \frac{\|\nabla u\|_{L^\infty(B_{R/4}(R\be_{n+1}))}}{u_\tau(R\be_{n+1})},
\]
where we are taking $R > 4$ and $\tau < 1$. By the interior Harnack inequality applied to $u_\tau$ (notice that $\Delta u_\tau = 0$ in $B_R(R\be_{n+1})$), we also know that
\[
u_\tau(R\be_{n+1})\ge c \|u\|_{L^\infty(B_{R/2}(R\be_{n+1}))},
\]
for some $c$ depending only on $n$. Hence, 
\begin{equation}
\label{eq:boundgtau}
|g_\tau(R)-1|\le C \tau \frac{\|\nabla u\|_{L^\infty(B_{R/4}(R\be_{n+1}))}}{\|u\|_{L^\infty(B_{R/2}(R\be_{n+1}))}} \le \frac{ \tilde C \tau}{R}
\end{equation}
for some constant $\tilde C$ depending only on $n$, where in the last inequality we are applying gradient estimates for harmonic functions in a ball of radius $R/2$. 

On the other hand, from \eqref{eq:fromhere} we can define a function
\[
w :=  u - \frac{g_\tau(R)}{C} u_\tau,
\]
that satisfies
\[
\left\{
\begin{array}{rcll}
\Delta w & \le & 0 & \quad\text{in}\quad \R^{n+1}\setminus \Lambda\\
w & \ge  & 0  &\quad\text{in}\quad B_R\\
w & = & 0 & \quad\text{on}\quad \Lambda.
\end{array}
\right.
\]
We can therefore apply again the boundary Harnack inequality in slit domains to deduce
\[
C w\ge \frac{w\left( \frac{R}{2}\be_{n+1}\right)}{u\left( \frac{R}{2}\be_{n+1}\right)}\,u = \left(1-\frac{1}{C} \frac{g_\tau(R)}{g_\tau(R/2)}\right) u \quad\text{in}\quad B_{R/2},
\]
for the same constant $C$ as in \eqref{eq:fromhere}. Rearranging terms with the definition of $w$, this implies
\[
g_\tau (R) u_\tau \le \left(C+ \frac{1}{C} \frac{g_\tau(R)}{g_\tau(R/2)}-1\right) u\quad\text{in}\quad B_{R/2}. 
\]
Observe that, from \eqref{eq:boundgtau}, if $R \ge R_0$ for some universal $R_0\ge 8\tilde C$, then 
\[
\frac{g_\tau(R)}{g_\tau(R/2)} \le \frac{1+\frac{\tilde C\tau}{R}}{1-\frac{2\tilde C\tau}{R}} = 1 + 3\frac{{\tilde C\tau}}{R-{2\tilde C\tau}} \le 1+4\tilde C \frac{\tau}{R}.
\]
 Hence we have 
\begin{equation}
\label{eq:tohere}
g_\tau (R) u_\tau \le \left(C+ \frac{1}{C} -1+C' \frac{\tau}{R}\right) u\quad\text{in}\quad B_{R/2}
\end{equation}
for some constant $C'$ depending only on $n$. Thus, we have gone from \eqref{eq:fromhere} to \eqref{eq:tohere}, where the constant is improved (if $R$ is large enough). Iterating the procedure, we have that 
\begin{equation}
\label{eq:Ck0}
g_\tau (R) u_\tau \le C_k u\quad\text{in}\quad B_{2^{-k}R},
\end{equation}
for all $k\in \N$ such that $2^{-k}R \ge R_0$, and where $C_k$ satisfy the recurrence relation:
\begin{equation}
\label{eq:Ck}
C_0 = C,\qquad C_{k+1} = C+ \frac{1}{C} -1+C' \frac{\tau}{2^{-k+1}R}.
\end{equation}

Now let $\alpha > 0$ be fixed, and let us consider the recurrence 
\[
x^\alpha_0 = C,\qquad x^\alpha_{k+1} = x^\alpha_k +\frac{1}{x^\alpha_k} -1+ \alpha\quad\text{for}\quad k\in \N. 
\]
Then if $C > \frac{1}{1-\alpha}$, $x^\alpha_k$ is decreasing and with limit $\frac{1}{1-\alpha}$. So for any $0<\alpha < \frac14$, assuming $C$ is large enough, there exists some $k_\alpha$ such that $x_k^\alpha \le 1+2\alpha$ for all $k \ge k_\alpha$. 

Fix $\alpha \in (0, 1/4)$. The constant $C > 0$ comes from \eqref{eq:fromhere} and we can always take it bigger so that $C > \frac{1}{1-\alpha}$. This fixes $C' > 0$ as in \eqref{eq:tohere}. Let $R \ge R_\alpha$ with $R_\alpha$ such that
\[
\frac{C'}{2^{-k_\alpha+1}R_\alpha} \le \alpha. 
\]
Then, from \eqref{eq:Ck} (recall $\tau < 1$) we know $C_k \le x_k^\alpha$ for $k \leq k_\alpha$, and thus $C_{k_\alpha}\le x_{k_\alpha}^\alpha \le 1+2\alpha$. Hence, in \eqref{eq:Ck0} we have
\[
\left(1-\frac{\tilde C\tau}{R}\right) u_\tau \le (1+2\alpha) u\quad\text{in}\quad B_{2^{-k_\alpha}R}.
\]
We now let first $R\to \infty$, and then $\alpha \downarrow 0$, to get 
\[
u_\tau = u(\cdot-\tau \be) \le u \quad\text{in}\quad \R^{n+1}. 
\]
Since $0<\tau < 1$ is arbitrary, this implies that $u$ is monotone in the direction $\be$, as we wanted to see. 
\end{proof}

\begin{rem}
\label{rem:alt_proof}
For the thin one-phase problem \eqref{eq:visc_sol_s}, such monotonicity follows directly from Lemma \ref{lem:viscC1/2} and the scaling of the problem:

If we assume $u\in\mathcal{G}_s(\be)$ with $\{u=0\}\neq\emptyset$ and define $v(x)=[u(x)-u(x-\tau\be)]/\tau$, then the  graphicality  of the free boundary implies that $\Delta v^-\ge 0$ in $\R^{n+1}$. As a consequence, we have
$$
v^-(0)\le C\left(\ave_{B_R}v^2\right)^{1/2}\le C\left(\frac{1}{\tau} \int_0^\tau \left[\ave_{B_R} u^2_{\be}(\cdot - t\be )\right]\, dt\right)^{1/2}\le CR^{-1/2},
$$
where the last inequality follows from the Caccioppoli estimate for the subharmonic function $u$ (where its growth is controlled by Lemma~\ref{lem:viscC1/2}). 

Sending $R\to\infty$ gives the desired monotonicity. 
\end{rem}

As in the local case, compactness of ``monotone" solutions follows immediately:

\begin{cor}
\label{cor:closedgraph2}
If $u_i \in \mathcal G_s(e)$ with $u_i \rightarrow u_\infty$ uniformly on compact sets, then $u_\infty \in \mathcal G_s(e)$. 
\end{cor}

We are now ready to show, in the thin setting, that global monotone solutions are actually minimizers of the functional \eqref{eq:JOm_frac}. As in the local case, we believe this to be a contribution of independent interest. 
 
For $R> 0$ we define 
$$B''_R := \{x'\in \mathbb R^{n-1}: |x'| < R\}.$$

\begin{prop}
\label{prop:cont_graph_s}
For $L > H > 0$, let $u$ be a viscosity solution to the thin one-phase problem \eqref{eq:visc_sol_s} in $\Omega \equiv B_2''\times (-2L-H, 2L + H) \times (-2L, 2L)$ with 
\[\partial_n u \geq 0\quad \text{in}\quad \Omega.
\]
 If its contact set is a subgraph
 \[\Lambda(u)  = \{(x', x_n,0)  : x_n \leq f(x')\}\]
  for some continuous function $f$ satisfying 
  \[
  -H < f(x') < H\quad\text{in}\quad B_2'',
  \] 
  then $u$ is the unique minimizer of the thin Alt-Caffarelli functional \eqref{eq:JOm_frac} in $D:= B_1''\times (-L, L)\times (-L, L)$. 
\end{prop}

As in the classic case, we split this proposition up into two pieces:

\begin{lem}\label{lem:minbelow}
Under the same assumptions and using the same notation as Proposition \ref{prop:cont_graph_s}, if $w$ is a minimizer to the functional \eqref{eq:JOm_frac} in $D$ with $w=u$ on $\partial D$ then $u\geq w$ in $D$. 
\end{lem}

The proof of this lemma follows the same scheme as its local  counterpart (Lemma~\ref{lem:cont_graph1}) but, as mentioned in the introduction, there is no known thin analogue to the results of Chang-Lara--Savin \cite{CS19}. Instead, we use a growth result whose proof we defer to Section \ref{sec:8} (see Theorem \ref{BoundaryRegularityIntro}):

\begin{proof}
We assume not, so that, for some $x_\circ\in D'\times\{0\}$ and $\eta_\circ > 0$, 
\[
w(x_\circ) > u(x_\circ) + \eta_\circ. 
\]
(Observe that such an $x_\circ$ exists on $\{y = 0\}$ by the maximum principle applied on the domain $D_+:=B_1''\times (-L, L)\times(0, L)$.)  We define for $\tau\in \R$,
\[
u_\tau(x'', x_n,y) := u(x'', x_n + \tau,y),
\]
and by the continuity of $u$ we can pick $s > 0$ small and fixed such that $w(x_\circ) > u_s(x_\circ) + \frac12 \eta_\circ$. 

\medskip

\noindent \emph{Step 1: Setting up the inf-convolution}\newline
\indent By the monotonicity of $u$ in the $\be_n$ direction (Proposition~\ref{thm:monotone_frac}) and the uniform continuity of the free boundary in $D$, there exists a set $E\subset\{y = 0\}$ such that
\[
\{u > 0\}\cap (D'\times\{0\}) \Subset E\Subset \{u_{s} > 0\}\cap \{y = 0\},
\]
where the compact inclusions need to be understood in the induced topology of $\{y = 0\}$. By the strong maximum principle, 
\[
\inf_E \partial_n u_s > 0
\]
and  there exists some $\delta  >0$ small such that 
\begin{equation}
\label{eq:from3_s}
u_s \ge u + \delta \quad\text{on}\quad \overline{\left\{u > 0\right\}\cap (D'\times \{0\})}.
\end{equation}
Let $\rho > 0$ be small enough, to be determined later and denote the inf-convolution of $u_\tau$, as in Definition \ref{Definitionthinsupconv}, by 
\[
v_\tau(x,y) := \inf_{z\in B_\rho'(x)} u_\tau(z,y).
\]
By monotonicity, we have $v_{s+t}(x) \geq v_s(x)$ for all $t \geq 0$ and $x\in \overline{D}$. Picking $\rho$ small enough, depending on $\eta_\circ$ above, $\delta > 0$ from \eqref{eq:from3_s}, and the (uniform) modulus of continuity of $u_s$, we have
\begin{equation}
\label{eq:fromfrom_s}
v_s(x) \ge u(x) \quad\text{for all}~~x\in  \overline{ D },\quad\text{and}\quad w(x_\circ) > v_s(x_\circ) + \frac14 \eta_\circ.
\end{equation}

\noindent \emph{Step 2: Initializing the sliding argument}\newline
\indent As in the local case, our $v_s$ will be a family of supersolutions which we will ``slide" down until we touch $w$ and get a contradiction. We start by showing that for $t$ large we have $v_{s+t} \geq w$. Indeed, if $t > \rho + H + L -s$ then $v_{s+t} > 0$ in $\overline{D}$ and
\[
\Delta v_{s+t} \le 0\quad\text{in}\quad \overline{D}\qquad\text{if}\quad t > \rho + H + L -s
\]
(in the viscosity sense). On the other hand, for all $t \ge 0$, 
\[
v_{s+t} \ge  u= w\quad\text{on}\quad\partial D, 
\]
so that by maximum principle (since $\Delta w \ge 0$ in $D$) 
\[
v_{s+t} \ge w \quad\text{in}\quad \overline{D}\qquad\text{if}\quad t > \rho + H + L -s. 
\]

Let us define now 
\[
t_* := \inf\{t \ge 0 : v_{s+t}\ge w \text{ in }\overline{D}\} =  \inf\{t \ge 0 : v_{s+t}\ge w \text{ on }\overline{D'\times \{0\}}\}, 
\]
where the second equality follows from the maximum principle (applied in $D\cap \{y > 0\}$).  By \eqref{eq:fromfrom_s}, $t_* > 0$.

\medskip

\noindent \emph{Step 3: The contact point in the sliding argument}\newline
\indent From \eqref{eq:from3_s} (taking $\rho$ smaller if necessary, depending only on~$s$) for any $t > 0$
\begin{equation}
\label{eq:from4_s}
v_{s+t} \ge v_s \geq u+\frac{\delta}{2} = w + \frac{\delta}{2} \quad\text{on}\quad \overline{\left\{u > 0\right\}\cap \partial D\cap\{y = 0\}}.
\end{equation}
 By continuity and \eqref{eq:from4_s}, together with the monotonicity in the $\be_n$ direction, there exists a touching point $\bar x \in \overline{\{w > 0\}\cap (D'\times \{0\})}$, i.e. $
v_{s+t_*}(\bar x) = w(\bar x)
$. 

We claim that $v_{s+t^*}(\bar{x}) = w(\bar{x}) = 0$. Indeed if $\bar x\in \{w > 0\}\cap (D'\times\{0\})$, then $\Delta v_{s+t_*}(\bar x) \le 0 =  \Delta w(\bar x)$ with $v_{s+t_*}\ge w$ in $D$, contradicting the maximum principle in $\{w > 0\} \cap D$. Furthermore, $\bar x\notin \{w > 0\}\cap \partial D\cap \{y = 0\}$, by \eqref{eq:from4_s}. 

\medskip

\noindent \emph{Step 4: The contradiction}\newline
\indent This leaves us two cases to consider, either the touching point is inside $D'\times\{0\}$ or on $\partial D'\times \{0\}$. 

If $\bar x\in F(w)\cap (D'\times \{0\})$, we can proceed as in the proof of Lemma~\ref{lem:cont_graph1} to say that $\partial \{w > 0\}$ has an exterior touching ball at $\bar{x}$ and thus $\bar{x}$ is a regular point (c.f. \cite[Proposition 5.10]{EKPSS20}). 

By the free boundary condition for minimizers,
\[
w(x,y) \ge U((x-\bar{x})\cdot \nu, y) +o(|(x,y)-(\bar x,0)|^{\frac12}),\]
where $\nu$ is the inward pointing unit normal to the ball at $\bar{x}$. In the other direction, Lemma \ref{PropertyOfConvolutionsThin} implies
\[
v_{s+t_*}(x) \le  U((x-\bar{x})\cdot \nu, y) +o(|(x,y)-(\bar x,0)|^{\frac12})
\]
This contradicts the nonlocal Hopf's lemma in this interior touching ball (see \cite[Proposition 4.11]{CS14}). 

So we are left to consider the case $\bar x\in \overline{F(w)\cap (D'\times \{0\})}\cap (\partial D'\times \{0\})$. Since $v_{s+t_*}\ge w$ in $\overline{D}$,   we also have $\bar x\in \partial_{\mathbb R^n}\{v_{s+t_*} > 0\}$. From \eqref{eq:from4_s}, $\bar x\notin F(u)$ and thus there is a neighborhood,  $Z\subset \partial D'$, of $\bar x $ where $u= w = 0$ on $Z\times\{0\}$. 

By our assumption that the contact set for $u$ is the subgraph of a continuous function in $\Omega \Supset \overline{D}$,  there is a small $\theta > 0$ such that $B'_{4\theta}(\bar x)\times\{0\} \subset \Lambda(u)\cap \Omega$. Using the harmonicity of $u$ in $\Omega \cap \{y > 0\}$ we can assume that $|u(x,y)| \le C_2|y|$ in $B_{2\theta}(\bar x)$. From this we can first conclude that $w=0$ on $(B'_{2\theta}(\bar x) \cap \partial D')\times \{0\}$ and that $|w| \leq C_2|y|$ on $B_{2\theta}(\bar x)\cap \partial D$ so we can invoke Theorem~\ref{thm:nondeg2} to get
\begin{equation}\label{e:wrhonondeg}
\sup_{B_r(\bar x)\cap D} w \ge c_1 r^{\frac12}\quad\text{for all}\quad r\in (0, \theta). 
\end{equation}

Furthermore, if $\varphi$ solves the boundary value problem 
\[
\left\{
\begin{array}{rcll}
\Delta \varphi & = & 0 & \quad\text{in}\quad B_\theta(\bar x) \cap D\\
\varphi & = & C_2 |y| & \quad\text{on}\quad B_\theta(\bar x) \cap \partial D\\
 \varphi & = & 1 & \quad\text{on}\quad \partial B_\theta(\bar x) \cap D,
\end{array}
\right.
\]
then $w \leq C\varphi$ in $B_\theta(\bar x)\cap D$ by the maximum principle (recall that $w$ is a minimizer so $\Delta w \geq 0$).

Note that the values of $\varphi$ are Lipschitz continuous in $B_{3\theta/4}(\bar x)\cap \partial D$, so we may invoke boundary Schauder estimates for harmonic functions to conclude that for any $\eps > 0$ there exists $c_\eps > 0$ such that
\begin{equation}
\label{eq:phiboundabove}
\varphi(x) = |\varphi(x) - \varphi(\bar x)| \le C_\eps |x-\bar x|^{1-\eps}\quad\text{in}\quad B_{\theta/2}(\bar x).
\end{equation}

 Taking $\eps = \frac14$ in \eqref{eq:phiboundabove} and invoking \eqref{e:wrhonondeg} we get,
\[
c_1 r^{\frac12}\le \sup_{B_r(\bar x)\cap D} w \le \sup_{B_r(\bar x)\cap D} \varphi \le C_{\frac14} r^{\frac34}, 
\]
for all $r\in (0, \theta/2)$, which is a contradiction.
\end{proof}

The bound from the other direction proceeds exactly as it does in the local case (Lemma \ref{lem:cont_graph2}), and as such we will simply state the thin result without proof:

\begin{lem}\label{lem:minabove}
Under the same assumptions and using the same notation as Proposition \ref{prop:cont_graph_s} if $w$ is a minimizer to the functional \eqref{eq:JOm_frac} in $D$ with $w=u$ on $\partial D$ then $u\leq w$ in $D$. 
\end{lem}

Thus,  we obtain:
\begin{proof}[Proof of Proposition~\ref{prop:cont_graph_s}]
It is a combination of Lemmas~\ref{lem:minbelow} and \ref{lem:minabove}. 
\end{proof}

And:

\begin{proof}[Proof of Theorem~\ref{MinimizingThinIntro}]
It follows from Proposition \ref{thm:monotone_frac} and Proposition \ref{prop:cont_graph_s}
\end{proof}

\section{Flatness of graphical solutions: the thin case} 
\label{sec:7}

Let us now show an analogous result to Lemma~\ref{lem:interm_lem} in the nonlocal setting. Recall that this lemma showed that if $F(u)$ is smooth an $u$ is monotone in a direction then either that direction is transverse to $F(u)$ at every point or $u$ is independent of that direction:

In the following lemma we consider $\Omega'\subset \R^n$ to be a smooth domain (i.e. $C^{2,\alpha}$), and $\Omega = \Omega'\times\{0\}$.  Abusing notation, we let $\partial \Omega$ denote the boundary of $\Omega'$ inside of $\R^n$ and denote by $\nu(x)\in \mathbb{S}^{n}_0$ the unit outward normal vector to $\Omega'$ at $x\in \partial\Omega$. Finally, in order to avoid the statement being empty, we assume that $\partial\Omega\cap B_1\cap\{y = 0\} \neq \emptyset$. 

\begin{lem}
\label{lem:interm_lem_frac}
Let $u$ be a viscosity solution to 
\begin{equation}
\label{eq:interm}
\left\{
\begin{array}{rcll}
\Delta u &=& 0&\quad\text{in}\quad B_1\setminus \Omega\\
u &=& 0&\quad\text{in}\quad \Omega\cap B_1\\
\lim_{x\to x_\circ} \frac{u}{d^{1/2}}(x) &=& 1&\quad\text{for}\quad x_\circ\in \partial \Omega \cap B_1, x\in \{y = 0\} \cap (B_1\setminus\Omega)\\
u &\ge& 0&\quad\text{in}\quad B_1.
\end{array}
\right.
\end{equation}
 We assume also that, for some $\be \in \mathbb{S}^{n}_0$,  
 $$\partial_\be u \ge 0\text{ in $B_1$.}$$ 

Then, either $\nu(x) \cdot \be > 0$ for all $x\in  \partial \Omega\cap B_1$, or $\partial_\be u \equiv 0$ in $B_1$.  
\end{lem}
\begin{proof}
By assumption, we immediately have $\nu(x) \cdot \be \ge 0$ for $x\in \partial \Omega$. Let us argue by contradiction, and so we may assume (up to a  rotation and translation) that $0\in \partial\Omega$ and $\be = \be_1$, with $\nu(0) = \be_n$ (so $\nu(0) \cdot \be = 0$). Let us also suppose $\partial_1 u \not\equiv 0$ in $B_1$. 

We denote by $\delta = \delta(x')$ the signed distance to $\partial \Omega$ inside of $\mathbb R^n$,
\[
\delta(x') = 
\left\{\begin{array}{rl}
{\rm dist}(x',\partial \Omega'),&\qquad\text{for any} \quad x'\in B_1\setminus\Omega',\\
-{\rm dist}(x', \partial \Omega'),&\qquad\text{for any}\quad x'\in \Omega'\cap B_1,
\end{array}
\right.
\]
whereas $r$ denotes the distance to $\partial\Omega$ in $B_1$, 
\[
r (x) = r((x', y)) = \sqrt{\delta^2(x') + y^2}, \qquad\text{for any}\quad x = (x', y) \in B_1.
\]
We define
\[
\tilde U := 2^{-1/2}(r+\delta)^{1/2}.
\]
Then, by \cite[Theorem 3.1]{DS15} we know that a solution to \eqref{eq:interm} can be expanded around a free boundary point (in this case, 0) as 
\begin{equation}
\label{eq:there}
\begin{split}
u(x', y) &  = \tilde U(x) \left(P_0(x', r) +O(|(x', r)|^{1+\alpha}\right)\\
& = \tilde U(x) \left( a^{(0)} + \boldsymbol{a}^{(1)}\cdot x' + a^{(2)} r + O(|(x', r)|^{1+\alpha})\right)  ,
\end{split}
\end{equation}
for some polynomial $P_0$ of degree 1. By the viscosity condition, we immediately have $a^{(0)} = c_*$. Moreover, by \cite[Theorem 3.1]{DS15}, we know 
\[
\frac{u(x',y)}{\tilde U(x)}\bigg|_{y = 0} \equiv \frac{u(x',0)}{\delta^{1/2}_+}(x') =: \eta(x') \in C^{1,\alpha}(\overline{B_1'\setminus \Omega'}).
\]
Since $\eta \equiv c_*$ on $\partial\Omega$, $\partial_i \eta(0) = 0$ for $1\le i \le n-1$  (recall $\nu(0) = \be_n$), and hence in \eqref{eq:there} we get $\boldsymbol{a}^{(1)}_i = 0$ for $1\le i \le n-1$.

On the other hand, by \cite[Theorem 3.1, eq. (3.4)]{DS15}, for $1\le i\le n-1$,
\[
\partial_i u (x) =  \frac{\tilde U(x)}{r}\left(P_{0,i}(x',r) + O(|(x', r)|^{1+\alpha})\right), 
\]
with 
\[
P_{0,i}(x', r) = r \boldsymbol{a}^{(1)}_i,\qquad\text{for}\quad 1\le i \le n-1.
\]
Combining the above,
\[
\partial_1 u(x) = \frac{\tilde U(x)}{r} O(|(x', r)|^{1+\alpha}).
\]

Let $h$ be harmonic outside of $\Lambda(u)$ with boundary values equal to $\partial_1 u$ on $\partial B_{3/4}$ and equal to $0$ on $\Lambda(u)$. By boundary Harnack for slit domains \cite[Corollary 3.4]{DS20} there exists a constant $c > 0$ such that $h \geq cu$ inside $B_{1/2}$. On the other hand, $\partial_1u \geq h$ in $B_{3/4}$ as both are harmonic in $B_{3/4}\backslash \Lambda(u)$ but $\partial_1 u \geq 0$ on $F(u)$. Using the expansion above this yields $$O(|(x', r)|^{1+\alpha}) \ge  r\left( c_*  + O(|(x', r)|)\right),$$ which gives a contradiction as $|(x',r)| \downarrow 0$.  
\end{proof}

As in the local setting, this tells us that homogeneous minimizers in low-dimensions are one-dimensional. 

\begin{prop}\label{p:lowdimonehomogenthin}
Let $u\in \mathcal G_s(\be_n)$ in $\mathbb R^{n+1}$ with 
\[
n \leq n^*_{\rm thin} + 1.
\]
 If $u$ is a homogeneous minimizer, then $$u = U((x\cdot \be), y),$$ for some $\be \in \mathbb S_0^{n}$ with $\be \cdot \be_n \geq 0$, and $U$ given by \eqref{HalfSpaceSolutionThin}.  
\end{prop}

\begin{proof}
The restriction on the dimension implies, by Lemma \ref{TraceSmoothThin}, that $\nu(x)$ exists and is a continuous function of $x\in F(u)\cap \mathbb S_0^n$. As in the classical setting, we consider $\be_\circ\in \mathbb{S}_0^{n}$ such that
\[
\be_\circ\in \argmin_{\bar\be \in \mathbb{S}_0^{n}} \left\{\bar\be \cdot\be_n : \bar\be \cdot \nu(x) \geq 0, \forall x\in \partial \Omega \cap \mathbb S_0^{n}\right\}.
\]
Arguing as in the local setting by minimality, there exists a $p_\circ\in F(u)\cap \mathbb{S}_0^{n}$ such that $ \be_\circ\cdot \nu(p_\circ) = 0$. If $\be_\circ = \be_n$ then we can invoke Lemma \ref{lem:interm_lem_frac} (recalling that $F(u)$ is smooth away from $0$  by Lemma \ref{TraceSmoothThin}) to conclude that $u$ is invariant in the direction $\be_\circ$. We do not need to worry about connectivity in $B_{1/2}(p_\circ)$ because of the assumption that $p_\circ\in F(u)$. Furthermore, the positivity set of any (nontrivial) global solution to the thin free boundary problem is connected.

If $\be_\circ \neq \be_n$, then arguing as in the classical case (but invoking Corollary \ref{cor:closedgraph2}) we see that $u\in \mathcal{G}_s(\be_\circ)$ and then Proposition~\ref{thm:monotone_frac} implies that $u$ is monotone in the direction $\be_\circ$. We again apply Lemma \ref{lem:interm_lem_frac} to conclude that $u$ is invariant in the direction $\be_\circ$. 


In either case, restricting $u$ to the space perpendicular to $\be_\circ$ gives a minimizing cone in $\mathbb R^{n}$. The definition of $n^*_{\rm thin}$ implies that $u$ is a one-dimensional solution. 
\end{proof} 

Finally, a blow-down argument shows us that global minimizers in low dimensions with graphical free boundaries are one dimensional. 

\begin{cor}\label{globalflat}
Let $u \in \mathcal G_s(\be_n)$ in $\mathbb R^{n+1}$ with \[
n \leq n_{\rm thin}^*+1.
\]
 If $u$ is a global minimizer to the thin Alt--Caffarelli functional \eqref{eq:JOm_frac},  then 
 \[
 u = U(x\cdot\be, y)
 \] for some $\be \in \mathbb S_0^n$ with $\be \cdot \be_n \geq 0$, and $U$ given by \eqref{HalfSpaceSolutionThin}.  
\end{cor} 

\begin{proof}
Consider the blow-down $$u_R(X) = \frac{u(RX)}{R^{1/2}},$$ as $R\uparrow \infty$. By Lemma \ref{lem:nonlocalblowdown} we have $$u_{R_i} \rightarrow v\qquad \text{locally uniformly}$$ along some subsequence $R_i$ where $v$ is a homogeneous minimizer to the thin one-phase problem. By Proposition \ref{p:lowdimonehomogenthin} we have that $v = U((x\cdot \be), y)$. Furthermore $\be \cdot \be_n \geq 0$ since $v\in \mathcal G_s(\be_n)$ by Corollary \ref{cor:closedgraph2}. 

Arguing as in the proof of Corollary \ref{cor:main000}, we can now apply the improvement of flatness Lemma \ref{l:improvethinflat} to conclude. 
\end{proof}

From this result the main theorem in the thin setting follows immediately: 

\begin{proof}[Proof of Theorem \ref{thm:mainThin}]
By Proposition~\ref{prop:cont_graph_s} $v$ is a globally defined minimizer to the thin one phase functional \eqref{eq:JOm_frac}. The theorem then follows after invoking Corollary \ref{globalflat}
\end{proof}

\section{Boundary growth near the fixed boundary for minimizers to the thin functional}
\label{sec:8}
This section is devoted to proving the key growth result we need to complete the proof that all monotone viscosity solutions to the thin problem are minimizers. 

Throughout this section, $w$ is a minimizer of the thin one-phase energy $\mathcal{J}^0_\Omega$ in the domain $\Omega = B_1\cap \{x_1 \ge 0\}\subset \R^{n+1} $. We denote by $\Gamma$ its free boundary, which is the boundary of $\{w > 0\}$ in the relative topology of $\{y = 0\}$. In particular, $\Gamma = \Gamma'\times\{0\}$. 

We will often use that if $w$ is minimizer, then $w_r(x) := \frac{w(rx)}{r^{1/2}}$ is a minimizer with its own boundary data as well. 

\begin{lem}[Nondegeneracy]
\label{lem:nondeg}
Let $w$ and $\Gamma$ as above. Let us suppose that $w = 0$ on $\{x_1 = y = 0\}\cap B_1$. Then
\[
w(x', 0) \ge c\, {\rm dist}^{\frac12} (x', \Gamma')\quad\text{for any}\quad x' \in B_{1/2}'\cap \{x_1 \ge 0\},
\]
for some $c$ depending only on $n$. 
\end{lem}
\begin{proof}
If we denote $r(x') := {\rm dist} (x', \Gamma')$ and since $w = 0$ on  $\{x_1 = y = 0\}\cap B_1$, we always have that $B_{r(x')}((x',0))\subset \{w > 0\}\cap B_1\cap \{x_1\ge 0\}$ for $x'\in \{w > 0 \}\cap B_{1/2}'\cap \{x_1\ge 0\}$. 
The proof  now follows as in \cite[Theorem 1.2]{CRS10}.

Indeed, after a rescaling it is enough to show that if $x = (x', 0)$ is at distance~1 from the free boundary then $w(x) = \eps$ cannot be arbitrarily small. By the Harnack inequality we know that $C_0^{-1}\eps \le w \le C_0\eps$ in $B_{1/2}(x)$, and by defining $\varphi$ to be a smooth nonnegative function such that $\varphi = 0$ in $B_{1/4}(x)$ and $\varphi = 2C_0$ in $B_{1/2}(x)\setminus B_{1/3}(x)$ we have that 
\[
v = \min\{w, \eps \varphi\}
\]
is an admissible competitor for $w$ in $B_{1/2}(x)$. 

Then,
\[
\int_{B_{1/2}(x)} |\nabla v|^2 \le \int_{B_{1/2}(x)} |\nabla w|^2 + C\eps^2,
\]
and 
\[
\mathcal{H}^n\left(\{v > 0\}\cap  B'_{1/2}(x)\right)\le \mathcal{H}^n\left(\{w > 0\}\cap B'_{1/2}(x)\right) - \mathcal{H}^n\left(B'_{1/4}(x)\right).
\]
Consequently, if $\eps > 0$ small enough depending only on $n$, we have 
\[
\mathcal{J}^0_{B_{1/2}(x)}(v) < \mathcal{J}^0_{B_{1/2}(x)}(w)
\]
which is a contradiction with the minimality of $w$.
\end{proof}

On the other hand, we also have the following result on the optimal regularity of $w$.   

\begin{thm}[Optimal regularity]
\label{thm:optimal_regularity}
Let $w$ as above with $w = \psi$ on $\{x_1 = 0\}$, where $\psi\in C^{1/2}(\{x_1 = 0\})$ and $\psi \equiv 0$ on $\{y = 0\}$. Then $w\in C^{1/2}(B_{1/2}\cap \{x_1\ge 0\})$ with 
\[
\|w\|_{C^{1/2}(B_{1/2}\cap \{x_1\ge 0\})}\le C\left(\|\psi\|_{C^{1/2}(B_1\cap \{x_1 = 0\})}+\|w\|_{L^\infty(B_1\cap \{x_1\ge0\})}  + 1 \right)
\]
for some $C$ depending only on $n$. 
\end{thm}
\begin{proof}
Let us denote 
\[
C_\circ = \|\psi\|_{C^{1/2}(B_1\cap \{x_1 = 0\})}+\|w\|_{L^\infty(B_1\cap \{x_1\ge0\})}.
\]

Let $\varphi$ be the solution to 
\[
\left\{
\begin{array}{rcll}
\Delta \varphi & = & 0& \quad\text{in}\quad B_1\cap\{x_1\ge 0\}\\
\varphi & = & |y|^{\frac12}&\quad\text{on}\quad B_1\cap \{x_1 = 0\}\\
\varphi & = & 1&\quad\text{on}\quad \partial B_1\cap \{x_1 \ge 0\}.
\end{array}
\right.
\]

Then by the regularity up to the boundary for harmonic functions with H\"older boundary data, we have that
\[
\varphi(x) \le Cx_1^{\frac{1}{2}} + y^{\frac{1}{2}} \leq C(x_1^2+y^2)^{\frac14}\quad\text{in}\quad B_{1/2}\cap \{x_1\ge 0\},
\]
for some $C$ depending only on $n$. Since $\Delta w \ge 0$ (as $w$ minimizes the energy), by comparison principle we have that 
\begin{equation}
\label{eq:barrierfromabove}
w(x) \le C C_\circ (x_1^2+y^2)^{\frac14}\quad\text{in}\quad B_{1/2}\cap \{x_1\ge 0\}.
\end{equation}
Let us first show our estimate on the thin space:

Let $z_1, z_2\in B_{1/2}\cap \{y = 0\}\cap \{x_1\ge 0\}$, and let us denote 
\[
r_i := {\rm dist}(z_i, \Gamma)  = {\rm dist}(z_i, \bar z_i),\quad \rho_i := (z_i)_1 = {\rm dist}(z_i, \{x_1 = 0\}),\quad\text{for} \ i = 1,2,
\]
where $\Gamma$ is the free boundary of $w$ on $\{y = 0\}$, and $\bar z_1, \bar z_2\in \Gamma$ are the corresponding projections on the free boundary. Let us denote $\delta := |z_1 - z_2|$. We split into two cases:
\begin{itemize}
\item If $5 \,r_1\le \rho_1$, then $B_{4\rho_1/5}(\bar z_1)\subset \{0 < x_1 <  2\rho_1\}$, and since $w$ is a minimizer in $B_{4\rho_1/5}(\bar z_1)$ with $\bar z_1$ a free boundary point, we can invoke universal $C^{1/2}$ estimates for minimizers of the thin problem (see \cite[Remark 7.10]{EKPSS20})  to get \[
\|w\|_{C^{1/2}(B_{\rho_1/2}(\bar z_1))}\le C,
\]
and hence if $\delta \le \rho_1/2$ we get $|w(z_1) - w(z_2)|\le C \delta^{\frac12}$. On the other hand, if $\delta \ge \rho_1/2$ we have $w(z_1)\le CC_\circ \delta^{\frac12}$ by \eqref{eq:barrierfromabove} and $\rho_2\le \rho_1+\delta \le 3\delta$, so that
$
w(z_2) \le C \rho_2^\frac12 \le C C_\circ\delta^\frac12
$
again by \eqref{eq:barrierfromabove}. In all cases we get $|w(z_1)-w(z_2)|\le C(1+C_\circ) \delta^\frac12$.

\item If $5\,r_1\ge \rho_1$, $\Delta w = 0$ in $B_{r_1}(z_1)\cap \{x_1 \ge 0\}$ with $C^{\frac12}$ boundary datum on $\{x_1 = 0\}$. Combined with the fact that $w \le CC_\circ r_1^\frac12$ in $B_{r_1}(z_1)\cap \{x_1\ge 0\}$ by \eqref{eq:barrierfromabove}, estimates for the Laplace equation yield 
\[
\|w\|_{C^{1/2}(B_{r_1/2}(z_1)\cap \{x_1\ge 0\})} \le C C_\circ,
\]
and so if $\delta < r_1/2$ we are done. On the other hand, if $\delta \ge r_1/2$ we have $w(z_1)\le C \delta^\frac12$ directly from \eqref{eq:barrierfromabove}, and as above we have $\rho_2 \le \rho_1+\delta \le 11\delta$ and so
$
w(z_2) \le  CC_\circ \delta^\frac12
$
thanks to \eqref{eq:barrierfromabove}.
\end{itemize}

From the estimates on the thin space we   obtain our desired estimate in $B_{1/2}\cap \{x_1\ge 0\}$ by standard techniques using boundary estimates and the fact that we have a barrier in \eqref{eq:barrierfromabove}. Indeed, to obtain the result in $\{y > 0\}$ it is enough to show 
\[
{\rm osc}_{B_r(z)} w \le C (1+C_\circ) r^\frac12\quad\text{for all}\quad B_{2r}(z)\subset B_{1/2}\cap \{x_1\ge 0\}\cap \{y \ge 0\}
\]
(see, for example, \cite[Appendix A]{FR22}). Now, given any $z\in B_{1/2}\cap \{x_1\ge 0\}\cap \{y \ge 0\}$, let us suppose $\rho_1 := z_1 \le z_{n+1}$ (the other case is symmetric). From the above observation, it is enough to see that $\|w\|_{C^{1/2}(B_{\rho_1/2}(z_1))}\le C (1+C_\circ)$, and this follows from boundary estimates for the Laplace equation together with the barrier \eqref{eq:barrierfromabove}. 
\end{proof}

Finally thanks to the previous considerations we have a second nondegeneracy type result:
\begin{thm}[Nondegeneracy in the full space near the fixed boundary]
\label{thm:nondeg2}
Under the hypotheses from Theorem~\ref{thm:optimal_regularity}, let us assume, moreover, that $\psi(x) \le \omega(y)|y|^\frac12$ for some modulus of continuity $\omega$. Let $\Gamma$ be the free boundary of $w$. Then, for any $z \in \Gamma\cap B_{1/4}\cap \{x_1\ge 0\}$ we have 
\[
\sup_{B_r(z)\cap \{x_1\ge 0\}} w \ge c r^\frac12\quad\text{for all}\quad r \in (0, 1/2),
\]
for some $c>0$ depending only on $n$ and $\omega$. 
\end{thm}

In order to prove this estimate we first show the following lemma:
\begin{lem}
\label{lem:interm}
Let $u$ be a minimizer of the thin one-phase energy in $B_{2Mr}\cap \{x_1\ge -\lambda\}$ for some $M, r> 0$, $\lambda \ge 0$, such that $0\in F(u)$ and 
\[
\|u \|_{C^{1/2}(B_{2Mr})\cap \{x_1\ge -\lambda\}}\le C_*
\]
for some $C_* > 0$. Let us suppose, moreover, that $0 \le u \le \omega(y)|y|^{1/2}$ on $\{x_1 = -\lambda\}$ for some modulus of continuity $\omega$, and that $B_r'(  z)\subset \{u > 0\}\cap \{x_1\ge -\lambda\}$ for some $| z| = r$. 

There exists $r_\circ$,  $\delta_\circ$, and $M$ such that if $r < r_\circ$ then 
\[
\sup_{B_{Mr}'\cap \{x_1\ge -\lambda\}} u \ge (1+\delta_\circ) u( z),
\]
where the constants $r_\circ$, $\delta_\circ$, and $M$ depend only on $n$, $C_*$, and $\omega$. 
\end{lem}
\begin{proof}
The proof follows by contradiction, assuming instead that there is a sequence of functions $u_k$, which minimize the thin one-phase energy in $B_{2M_k r_k}\cap \{x_1 \ge -\lambda_k\}$ such that $0$ is a free boundary point for $u_k$,
\[
[u_k ]_{C^{1/2}(B_{2M_kr_k})\cap \{x_1\ge -\lambda_k\}}\le C_*,
\]
$0 \le u_k \le \omega(y)|y|^{1/2}$ on $\{x_1 = -\lambda_k\}$, $B_{r_k}'(  z_k)\subset \{u_k > 0\}\cap \{x_1\ge -\lambda_k\}$ for some $| z_k| = r_k$, but
\[
\lim_{k \to \infty} \frac{1}{u_k(z_k)} \sup_{B_{M_kr_k}'\cap \{x_1\ge -\lambda_k\}} u_k = 1,
\]
for some sequence $M_k \to \infty$ and $r_k \downarrow 0$.

If we define $\bar u_k(x) :=  \frac{u_k(r_k x)}{r_k^{1/2}}$, then $(\bar u_k)_{k\in \N}$ are a minimizers in $B_{2M_k}\cap \{x_1\ge -\lambda_k r_k^{-1}\}$ with 0 a free boundary point, satisfying
\[
[\bar u_k ]_{C^{1/2}(B_{2M_k})\cap \{x_1\ge -\lambda_k r_k^{-1}\}}\le C_*,
\]
$0 \le \bar u_k \le \omega(r_k y)|y|^{1/2}$ on $\{x_1 = -\lambda_k r_k^{-1}\}$, $B_{1}'(  \bar z_k)\subset \{\bar u_k > 0\}\cap \{x_1\ge -\lambda_kr_k^{-1}\}$ for some $| \bar z_k| = 1$, and
\[
\lim_{k \to \infty} \frac{1}{\bar u_k(\bar z_k)}\sup_{B_{M_k }'\cap \{x_1\ge -\lambda_kr_k^{-1}\}} \bar u_k = 1.
\]

In particular, thanks to Lemma~\ref{lem:nondeg} there exists some universal constant $c_*$ such that
\[
\bar u_k (\bar z_k) \ge c_*>0. 
\]

Since $\bar u_k(0)=0$, the uniform estimates on the $C^{1/2}$-seminorm in $B_{2M_k}$ allows us to apply Arzela-Ascoli to obtain that
\[
\bar u_k \to u_\infty\quad\text{locally uniformly in $\R^{n+1}$}.
\]

Up to a subsequence, we assume $\bar z_k \to \bar z_\infty$ with $|\bar z_\infty| = 1$, so that $B_1'(z_\infty) \subset \{u_\infty > 0\}$ (again invoking Lemma~\ref{lem:nondeg}). 
Furthermore, $\bar u_k (\bar z_k) \ge c_*>0$, so we know that $u_\infty(0)= 0$ but  $u_\infty \not \equiv 0$. Up to a subsequence, we can further assume $\lambda_k r_k^{-1}\to \lambda_\infty$ for some $0\le \lambda_\infty \le \infty$, so  $B_1'(\bar z_\infty) \subset \{u_\infty > 0\}\cap \{x_1\ge -\lambda_\infty\}$ (again thanks to Lemma~\ref{lem:nondeg}). In particular, $u_\infty$ is harmonic in $\{y \neq 0\}$ and in $B_1'(\bar z_\infty)$. 

Finally, we also have $u_\infty = 0$ on $\{x_1 = -\lambda_\infty\}$ and $\bar z_\infty$ is a global maximum for $u_\infty$. By taking the odd reflection of $u_\infty$ with respect to $-\lambda_\infty$ (if $\lambda_\infty < \infty$), denoted $\bar u_\infty$, we have that $\bar u_\infty$ is a globally defined function, harmonic in $\{y\neq 0\}$ and on $\{y = 0\}\cap \{\bar u_\infty \neq 0\}$, with a global positive maximum at $\bar z_\infty$. This contradicts the maximum principle for the fractional Laplacian $(-\Delta)^\frac12$. 
\end{proof}

Using the previous lemma, we can now prove Theorem~\ref{thm:nondeg2}:
\begin{proof}[Proof of Theorem~\ref{thm:nondeg2}]
We follow the ideas of \cite{BCN90} (see also \cite{CRS10}). 

Let $r_\circ$, $\delta_\circ$, and $M$ be given by Lemma~\ref{lem:interm} with $\omega$ given by the statement and $C_*$ given by Theorem~\ref{thm:optimal_regularity}. Let $z\in \Gamma\cap B_{1/4}\cap \{x_1\ge 0\}$ as in the theorem statement. Translate $z$ to the origin  and we are in a situation where, as long as $2Mr \le \frac12$, we can apply Lemma~\ref{lem:interm} with $\lambda = -z_1$. 

Let $ 0<r <r_\circ$ be fixed. We construct inductively starting from $z_0\in B_{r/10}\cap \{w > 0\}$ a sequence of points $(z_k)_{k\in \N}$ with $z_k\in \{w > 0\}$ and $r_k := {\rm dist}(z_k, \Gamma) = {\rm dist}(z_k, \bar z_k)$ such that $z_{k+1}\in B'_{Mr_k}(\bar z_k)$ and 
\[
u(z_{k+1})\ge (1+\delta_\circ) u(z_k),
\]
thanks to Lemma~\ref{lem:interm}. Observe that 
\[
|z_{k+1}-z_k|\le (M+1) r_k
\]
and that $u(z_k)$ is increasing geometrically. We can do this as long as $B'_{Mr_k}(\bar z_k) \subset B'_{3/4}(-z)$

We denote by $k_\circ+1$ the first value of $k\in \N$ such that $z_{k}$ falls outside of $B_{r}$. If $r_{k_\circ}> \frac{r}{10(M+1)}$, then $u(z_{k_\circ}) \geq cr_{k_\circ}^{1/2} \geq c'r^{1/2}$ by Lemma \ref{lem:nondeg} and we are done. So we can assume that $r_{k_\circ} < \frac{r}{10(M+1)}$ which means that $|z_{k_{\circ}}| > r- (M+1)r_{k_\circ} > \frac{9r}{10}$. Also in this case $B'_{Mr_{k_\circ}}(\bar z_{k_\circ}) \subset B'_{3/4}(-z)$. 

Using the estimates above
\[
u(z_{k_\circ}) \ge \sum_{1 \le i \le k_\circ} \left( u(z_i) - u(z_{i-1})\right) \ge \delta_\circ \sum_{0 \le i \le k_\circ-1} u(z_{i}).
\]
Now, thanks to Lemma~\ref{lem:nondeg} we know $u(z_{i}) \ge c r_i^{1/2}$, and thus
\[
u(z_{k_\circ}) \ge \frac{\delta_\circ c}{(M+1)^{1/2}} \sum_{0 \le i \le k_\circ-1} |z_{i+1}-z_i|^{1/2} \ge c' \left|\sum_{0 \le i \le k_\circ-1} (z_{i+1}-z_i)\right|^{1/2}\ge c'' r^{1/2},
\]
for some $c''$ depending only on $n$ and $\omega$. That is, given $0< r < r_\circ$ we have found a point $z_{k_\circ}\in B_{r}$ such that $u(z_{k_{\circ}}) \ge c'' r^{1/2}$ for some $c''$. Since $r$ was arbitrary and $r_\circ$ and $M$ depend only on $n$ and $\omega$, we get the desired result. 
\end{proof}


\begin{thebibliography}{99}

\bibitem[AAC01]{AAC01}
	\newblock G. Alberti, L. Ambrosio, X. Cabr\'e,
	\newblock {\em On a long-standing conjecture of E. De Giorgi: symmetry in 3D for general nonlinearities and a local minimality property},
	\newblock Acta Appl. Math. 65 (2001), 9-33. 

\bibitem[All12]{All}
	\newblock M. Allen,
	\newblock {\em Separation of a lower dimensional free boundary in a two-phase problem},
	\newblock Math. Res. Lett. 19 (2012), 1055-1074. 

\bibitem[Alm66]{Alm66}
	\newblock F. J. Almgren Jr.,
	\newblock {\em Some interior regularity theorems for minimal surfaces and an extension of Bernstein's theorem},
	\newblock  Ann. of Math. 84 (1966), 277-292.

\bibitem[AC81]{AC81}
	\newblock H. Alt, L. Caffarelli, 
	\newblock {\em Existence and regularity for a minimum problem with free boundary},
	\newblock J. Reine Angew. Math 325 (1981), 105-144.
	
\bibitem[ACF82]{ACF82}
	\newblock H. Alt, L. Caffarelli,  A. Friedman,
	\newblock {\em Asymmetric jet flows},
	\newblock Comm. Pure Appl. Math. 35 (1982), 29-68.

\bibitem[ACF82b]{ACF82b}
	\newblock H. Alt, L. Caffarelli,  A. Friedman,
	\newblock {\em Jet flows with gravity},
	\newblock J. Reine Angew. Math. 331 (1982), 58-103.
	
\bibitem[ACF83]{ACF83}
	\newblock H. Alt, L. Caffarelli,  A. Friedman,
	\newblock {\em Axially symmetric jet flows},
	\newblock Arch. Ration. Mech. Anal. 81 (1983), 97-149.
	
\bibitem[AC00]{AC00}
	\newblock L. Ambrosio, X. Cabr\'e,
	\newblock {\em  Entire solutions of semilinear elliptic equations in $\R^3$ and a conjecture of De Giorgi},
	\newblock J. Amer. Math. Soc. 13 (2000), 725-739.

\bibitem[AS22]{AS22}
	\newblock A. Audrito, J. Serra,
	\newblock {\em Interface regularity for semilinear one-phase problems},
	\newblock Adv. Math. 403 (2022), 108380.

\bibitem[BSS76]{BSS76}
         \newblock G. Baker, P. Saffman, J. Sheffield,
         \newblock{\em Structure of a linear array of  hollow vortices of finite cross-section},
         \newblock J. Fluid Mech. 74 (1976), 469-476.

\bibitem[BCN90]{BCN90}
	\newblock H. Berestycki, L. Caffarelli, L. Nirenberg,
	\newblock {\em Uniform estimates for regularization of free
boundary problems},
	\newblock In: Analysis and Partial Differential Equations, C. Sadosky (ed.), Lecture
Notes in Pure Appl. Math. 122, Dekker, New York, 567-619 (1990).
	
\bibitem[Ber15]{Ber15}
	\newblock S. N. Bernstein,
	\newblock {\em Sur une th\'eor\`eme de g\'eometrie et ses applications aux \'equations d\'riv\'ees partielles du type elliptique},
	\newblock Comm. Soc. Math. Kharkov 15 (1915-17), 38-45.
	
\bibitem[BDG69]{BDG69}
	\newblock E. Bombieri, E. De Giorgi, E. Giusti,
	\newblock {\em Minimal cones and the Bernstein problem},
	\newblock Invent. Math. 7 (1969), 243-268.
	
\bibitem[BG72]{BG72}
	\newblock E. Bombieri, E. Giusti,
	\newblock{\em Harnack's inequality for elliptic differential equations on minimal surfaces},
	\newblock Invent. Math. 15 (1972), 24-46.
	
\bibitem[BL82]{BL82}
	\newblock J. D. Buckmaster, G. S. Ludford,
	\newblock {\em Theory of Laminar Flames},
	\newblock Cambridge Univ. Press, Cambridge, 1982.
	
	
\bibitem[CEF22]{CEF22}
	\newblock X. Cabr\'e, I. U. Erneta, J. C. Felipe-Navarro,
	\newblock {\em A Weierstrass extremal field theory for the fractional Laplacian}
	\newblock Preprint arXiv: 2211.16536. 
	        
\bibitem[CP18]{CP18}
	\newblock X. Cabr\'e, G. Poggesi,
	\newblock {\em Stable solutions to some elliptic problems: minimal cones, the Allen-Cahn equation, and blow-up solutions}
	\newblock Geometry of PDEs and related problems, 1-45, Lecture Notes in Math., 2220, Fond. CIME/CIME Found. Subser., Springer, Cham, 2018.
	
\bibitem[CS14]{CS14}
	\newblock X. Cabr\'e, Y. Sire,
	\newblock {\em Nonlinear equations for fractional Laplacians I: Regularity, maximum principles, and Hamiltonian estimates,}
	\newblock Ann. Inst. H. Poincare Anal. Non Lineaire 31 (2014), 23-53. 
	
\bibitem[Caf87]{Caf87}
	\newblock L. Caffarelli, 
	\newblock {\em A Harnack inequality approach to the regularity of free boundaries. I. Lipschitz free boundaries are $C^{1,\alpha}$},
	\newblock Rev. Mat. Iberoam. 3 (1987), 139-162.
		
\bibitem[Caf88]{Caf88}
	\newblock L. Caffarelli, 
	\newblock {\em A Harnack inequality approach to the regularity of free boundaries. III. Existence theory, compactness, and dependence on $X$},
	\newblock  Ann. Scuola Norm. Sup. Pisa Cl. Sci. 15 (1988), 583-602.
	
\bibitem[Caf89]{Caf89}
	\newblock L. Caffarelli, 
	\newblock {\em A Harnack inequality approach to the regularity of free boundaries. II. Flat free boundaries are Lipschitz},
	\newblock  Comm. Pure Appl. Math. 42 (1989), 55-78.      
       
       
\bibitem[CJK04]{CJK04}
	\newblock L. Caffarelli, D. Jerison, C. Kenig,
	\newblock {\em Global energy minimizers for free boundary problems
and full regularity in three dimensions},
	\newblock Noncompact problems at the intersection of geometry, analysis, and topology, 83-97, Contemp. Math. 350, Amer. Math. Soc., Providence, RI, 2004. 
		
\bibitem[CRS10]{CRS10}
	\newblock L. Caffarelli, J. Roquejoffre, Y. Sire,
	\newblock {\em Variational problems with free boundaries for the fractional Laplacian},
	\newblock J. Eur. Math. Soc. 12 (2010), 1151-1179.
	
\bibitem[CS05]{CS05}
	\newblock L. Caffarelli, S. Salsa, 
	\newblock {\em 	A Geometric Approach to Free Boundary Problems},
	\newblock Graduate Studies in Mathematics, 68. American Mathematical Society, Providence, RI, 2005. 
		
\bibitem[CV95]{CV95}
	\newblock L. Caffarelli, J. L. V\'azquez,
	\newblock {\em 	A free-boundary problem for the heat equation arising in flame propagation},
	\newblock Trans. Amer. Math. Soc. 347 (1995), 411-441.
	
\bibitem[CS19]{CS19}
	\newblock H. Chang-Lara, O. Savin,
	\newblock {\em 	Boundary regularity for the free boundary in the one-phase problem},
	\newblock 	New developments in the analysis of nonlocal operators, 149-165, Contemp. Math. 723, Amer. Math. Soc., Providence, RI, 2019.

\bibitem[CM11]{CM11}
	\newblock T. H. 	Colding, W. P. Minicozzi II,
	\newblock {\em A course in minimal surfaces},
	\newblock Graduate Studies in Mathematics, 121, American Mathematical Society, Providence, RI, 2011.
	
\bibitem[CG11]{CG11}
       \newblock D. Crowdy, C. Green, 
       \newblock{\em Analytical solutions for von K\'arm\'an streets of hollow vortices},
       \newblock Phys. Fluids 23 (2011), 126602.
	
\bibitem[DET19]{DET} 
	\newblock G. David, M. Engelstein, T. Toro,
	\newblock {\em Free boundary regularity for almost-minimizers},
	\newblock Adv. Math. 350  (2019), 1109-1192.
	
\bibitem[DT15]{DT15}
	\newblock G. David, T. Toro,
	\newblock {\em Regularity of almost minimizers with free boundary},
	\newblock Calc. Var. Partial Differential Equations 54 (2015), 455-524. 
	
\bibitem[DeG78]{DeG78}
	\newblock E. De Giorgi,
	\newblock {\em 	Convergence problems for functionals and operators},
	\newblock  Proc. Int. Meeting on Recent Methods in
Nonlinear Analysis (Rome, 1978), 131-188.

\bibitem[DeG65]{DeG65}
	\newblock E. De Giorgi,
	\newblock {\em 	Una estensione del teorema di Bernstein},
	\newblock  Ann. Scuola Norm. Sup. Pisa (3) 19 (1965), 79-85.
	
\bibitem[DeS11]{D11}
	\newblock D. De Silva,
	\newblock{\em Free boundary regularity for a problem with right hand side},
	\newblock Interfaces Free Boundaries 13 (2011), 223-238.

\bibitem[DJ09]{DJ09} 
	\newblock D. De Silva, D. Jerison, 
	\newblock \emph{A singular energy minimizing free boundary},
	\newblock J. Reine Angew. Math. 635 (2009), 1-22.
	
\bibitem[DJ11]{DJ11}
	\newblock D. De Silva, D. Jerison,
	\newblock {\em 	A gradient bound for free boundary graphs},
	\newblock Comm. Pure Appl. Math. 64 (2011), 538-555.
	
\bibitem[DJS22]{DJS22}
	\newblock D. De Silva, D. Jerison, H. Shahgholian,
	\newblock{\em Inhomogeneous global minimizers to the one-phase free boundary problem},
	\newblock Comm. Partial Differential Equations 47 (2022), 1193-1216.
	
\bibitem[DR12]{DR12}
	\newblock D. De Silva, J. Roquejoffre,
	\newblock{\em Regularity in a one-phase free boundary problem for the fractional Laplacian},
	\newblock Ann. Inst. H. Poincare Anal. Non Lineaire 29 (2012), 335-367.
	
\bibitem[DS12]{DS12}
	\newblock D. De Silva, O. Savin,
	\newblock {\em $C^{2,\alpha}$-regularity of flat free boundaries for the thin one-phase problem},
	\newblock J. Differential Equations  253  (2012), 2420-2459.

\bibitem[DS15]{DS15}
	\newblock D. De Silva, O. Savin,
	\newblock {\em $C^\infty$ regularity of certain thin free boundaries},
	\newblock Indiana Univ. Math. J. 64 (2015), 1575-1608.
	
\bibitem[DS15b]{DS15b}
	\newblock D. De Silva, O. Savin,
	\newblock{\em Regularity of Lipschitz free boundaries for the thin one-phase problem},
	\newblock J. Eur. Math. Soc. 17 (2015), 1293-1326.
	
	
\bibitem[DS20]{DS20}
	\newblock D. De Silva, O. Savin,
	\newblock {\em A short proof of boundary Harnack inequality},
	\newblock J. Differential Equations 269 (2020), 2419-2429.
	
\bibitem[DSS14]{DSS14}
	\newblock	D. De Silva, O. Savin, Y. Sire,
	\newblock {\em A one-phase problem for the fractional Laplacian: regularity of flat free boundaries},
	\newblock Bull. Inst. Math. Acad. Sin. (N.S.), 9 (2014), 111-145.

\bibitem[DKW11]{DKW11}
	\newblock M. Del Pino, M. Kowalczyk, J. Wei,
	\newblock {\em On De Giorgi's conjecture in dimension $N\ge 9$},
	\newblock  Ann. of Math.   174 (2011), 1485-1569.
	
\bibitem[EFW22]{EFW22}
        \newblock S. Eberle, A. Figalli, G. Weiss,
        \newblock{\em Complete classification of global solutions to the obstacle problem},
        \newblock Preprint arXiv: 2208.03108.  
        
   \bibitem[ESW23]{ESW20}
        \newblock S. Eberle, H. Shahgholian, G. Weiss,
        \newblock{\em On global solutions of the obstacle problem},
        \newblock Duke Math. J. 172 (2023), 2149-2193.
        
\bibitem[ESW22]{ESWPreprint}
        \newblock S. Eberle, H. Shahgholian, G. Weiss,
        \newblock{\em The structure of the regular part of the free boundary close to singularities in the obstacle problem},
        \newblock J. Differential Equations, to appear.
        
\bibitem[ERW21]{ERW21}
        \newblock S. Eberle, X. Ros-Oton, G. Weiss,
        \newblock{\em Characterizing compact coincidence sets in the thin obstacle problem and the obstacle problem for the fractional Laplacian},
        \newblock Nonlinear Anal. 211 (2021), Paper No. 112473.

\bibitem[EY23]{EY23}
        \newblock S. Eberle, H. Yu,
        \newblock{\em Compact compact sets of sub-quadratic solutions to the thin obstacle problem},
        \newblock Preprint arXiv: 2304.03939.
        
        \bibitem[EY23b]{EY23b}
        \newblock S. Eberle, H. Yu,
        \newblock{\em Solutions to the nonlinear obstacle problem with compact contact sets},
        \newblock Preprint arXiv: 2305.19963.

	
\bibitem[EE19]{EE}
	\newblock N. Edelen, M. Engelstein,
	\newblock {\em Quantitative stratification for some free boundary probems},
	\newblock Trans. Amer. Math. Soc. 371 (2019), 2043-2072.
	
\bibitem[ESV23]{ESV22}
	\newblock N. Edelen, L. Spolaor, B. Velichkov,
	\newblock{\em A strong maximum principle for minimizers of the one-phase Bernoulli problem},
	\newblock Indiana Univ. Math. J., to appear. 
	
\bibitem[EKPSS21]{EKPSS20}
	\newblock M. Engelstein, A. Kauranen, M. Prats, G. Sakellaris, Y. Sire,
	\newblock {\em Minimizers for the thin one-phase free boundary problem},
	\newblock Comm. Pure Appl. Math. 74 (2021), 1971-2022.
	
\bibitem[ESV20]{ESV20}
	\newblock M. Engelstein, L. Spolaor, B. Velichkov,
	\newblock {\em Uniqueness of the blowup at isolated singularities for the Alt-Caffarelli functional},
	\newblock Duke Math. J. 169 (2020), 1541-1601.
%
	

	
\bibitem[FR19]{FR19}
	\newblock X. Fern\'andez-Real, X. Ros-Oton, 
	\newblock {\em On global solutions to semilinear elliptic equations related to the one-phase free boundary problem},
	\newblock Discrete Contin. Dyn. Syst. A 39 (2019), 6945-6959.
	
	\bibitem[FR23]{FR20}
	\newblock X. Fern\'andez-Real, X. Ros-Oton, 
         \newblock{\em Stable cones in the thin one-phase problem},
         \newblock Amer. J. Math, to appear. 
	
	
\bibitem[FR22]{FR22}
	\newblock X. Fern\'andez-Real, X. Ros-Oton, 
	\newblock {\em Regularity Theory for Elliptic PDE},
	\newblock Zurich Lectures in Advanced Mathematics, EMS Press, 2022.
	
	
\bibitem[FY23]{FY23}
	\newblock X. Fern\'andez-Real, H. Yu, 
	\newblock {\em Generic properties in free boundary problems},
	\newblock Preprint arXiv: 2308.13209.
	
\bibitem[Fle62]{Fle62}
	\newblock W. H. Fleming, 
	\newblock {\em On the oriented Plateau problem},
	\newblock  Rend. Circ. Mat. Palermo (2) 11 (1962), 69-90.
	
\bibitem[GG98]{GG98}
	\newblock N. Ghoussoub, C. Gui, 
	\newblock {\em On a conjecture of De Giorgi and some related problems},
	\newblock  Math. Ann. 311 (1998), 481-491.
	
\bibitem[HHP11]{HHP11}
         \newblock L. Hauswirth, F. H\'elein, F. Pacard,
         \newblock{\em On an overdetermined elliptic problem},
         \newblock Pacific J. Math. 250 (2011), 319-334.
         
	

	
\bibitem[JK16]{JK16}
	\newblock D. Jerison, N. Kamburov,
	\newblock {\em Structure of one-phase free boundaries in the plane},
	\newblock Int. Math. Res. Not. 19 (2016), 5922.
	
\bibitem[JS15]{JS15}
	\newblock D. Jerison, O. Savin,
	\newblock {\em Some remarks on stability of cones for the one-phase free boundary problem},
	\newblock Geom. Funct. Anal. 25 (2015), 1240-1257.	
	
%

\bibitem[KW23]{KW22}
        \newblock N. Kamburov, K. Wang,
        \newblock{\em Nondegeneracy for stable solutions to the one-phase free boundary problem},
        \newblock Math. Ann., to appear. 

\bibitem[KLT13]{KLT13}
          \newblock D. Khavinson, E. Lundberg,  R. Teodorescu,
          \newblock{\em An overdetermined problem in potential theory},
          \newblock Pacific J. Math. 265 (2013), 85-111.
          
          
        	
	
          
\bibitem[LWW21]{LWW21}
      \newblock Y. Liu, K. Wang, J. Wei,
      \newblock{\em On smooth solutions to one-phase free boundary problem in $\R^n$},
      \newblock Int. Math. Res. Not. 20 (2021), 15682-15732.
%
	
	
\bibitem[Sav09]{Sav09} 
	\newblock O. Savin,
	\newblock \emph{Regularity of  at level sets in phase transitions},
	\newblock  Ann. of Math. 169 (2009), 41-78.

\bibitem[Sim68]{Sim68} 
	\newblock J. Simons,
	\newblock \emph{Minimal varieties in riemannian manifolds},
	\newblock  Ann. of Math. 88 (1968), 62-105.
 

\bibitem[RT21]{RT21} 
	\newblock X. Ros-Oton, C. Torres-Latorre,
	\newblock \emph{New boundary Harnack inequalities with right hand side},
	\newblock J. Differential Equations 288 (2021), 204-249.
	
\bibitem[Tra14]{T14}
          \newblock M. Traizet,
          \newblock{\em Classification of the solutions to an overdetermined elliptic problem in the plane},
          \newblock Geom. Funct. Anal. 24 (2014), 690-720.
 
\bibitem[Vel23]{Vel19} 
	\newblock B. Velichkov, 
	\newblock \emph{Regularity of the One-Phase Free Boundaries},
	\newblock Lecture Notes of the Unione Matematica Italiana, 28, Springer Cham 2023. 
	
\bibitem[Wei99]{Wei99}
	\newblock G. S. Weiss,
	\newblock \emph{Partial regularity for a minimum problem with free boundary},
	\newblock J.  Geom. Anal. 9 (1999), 317-326. 



\end{thebibliography}
\end{document}